\documentclass[a4paper,%
fontsize=11pt,%
oneside,%
numbers=enddot]{scrartcl}
\KOMAoptions{DIV=12}    
\usepackage[utf8]{inputenc}
\usepackage[T1]{fontenc}
\usepackage{graphicx}
\usepackage{textcomp}
\usepackage[fleqn]{amsmath}
\usepackage{amsthm}
\usepackage{amssymb}
\usepackage{amsfonts}
\usepackage{soul}
\usepackage{pifont}        
\usepackage{libertine}
\usepackage[libertine,
slantedGreek,
nosymbolsc,
nonewtxmathopt,
subscriptcorrection]{newtxmath}
\usepackage[scaled=0.95,varqu,varl]{inconsolata}
\usepackage[scr=boondox]{mathalpha}   
\usepackage{euscript}   
\usepackage{leftindex}
\allowdisplaybreaks[1]  
\numberwithin{equation}{section}    
\usepackage[svgnames,hyperref]{xcolor}
\definecolor{orng}{HTML}{F35400}
\definecolor{bleu}{HTML}{BCE6F2}
\definecolor{dblue}{HTML}{0455BF}
\definecolor{dgreen}{HTML}{02724A}
\definecolor{dgreen2}{HTML}{025951}
\definecolor{dred}{HTML}{D90404}
\definecolor{dviolet}{HTML}{42208C}
\definecolor{labelkey}{HTML}{025951}
\definecolor{refkey}{HTML}{025951}
\definecolor{refkey}{rgb}{0,0.6,0.0}
\definecolor{Brown}{rgb}{0.45,0.0,0.05}
\definecolor{dgreen}{rgb}{0.00,0.49,0.00}
\definecolor{dblue}{rgb}{0,0.18,0.75}
\definecolor{lblue}{rgb}{0,0.7,0.75}
\definecolor{dviolet}{HTML}{9400D3}
\definecolor{pblue}{rgb}{0.1176,0.5647,1}
\definecolor{nblue}{rgb}{0.2,0.3,1}
\definecolor{pgreen}{rgb}{0.1961,0.8039,0.1961}
\definecolor{ngreen}{rgb}{0.0,0.6,0.3}
\definecolor{pred}{rgb}{1.0,0.2706,0.0}
\definecolor{magenta}{HTML}{ff00ff}
\definecolor{hotmagenta}{rgb}{1.0, 0.11, 0.81}
\definecolor{dorng}{rgb}{0.91,0.41,0.17}
\definecolor{dgray}{rgb}{0.41,0.41,0.41}
\usepackage{tikz,tkz-euclide,pgfplots}
\usetikzlibrary{arrows,decorations}
\addtokomafont{section}{\centering}

\usepackage{enumitem}
\setlist{itemsep=-2.0pt}

\usepackage{upref}  
\usepackage{hyperref}
\hypersetup{colorlinks=true,
linktocpage=true,
linkcolor=dblue,
citecolor=dgreen,
urlcolor=dred,
pdfencoding=auto,
hypertexnames=false}
\makeatletter
\g@addto@macro\th@plain{
\thm@headfont{\bfseries\sffamily}
\thm@notefont{}}
\g@addto@macro\th@definition{
\thm@headfont{\bfseries\sffamily}
\thm@notefont{}}
\g@addto@macro\th@remark{
\thm@headfont{\bfseries\sffamily}
\thm@notefont{}}
\makeatother
\theoremstyle{plain}
\newtheorem{theorem}{Theorem}[section]

\newtheorem{corollary}[theorem]{Corollary}
\newtheorem{lemma}[theorem]{Lemma}

\theoremstyle{definition}

\newtheorem{example}[theorem]{Example}
\newtheorem{problem}[theorem]{Problem}

\theoremstyle{remark}
\newtheorem{remark}[theorem]{Remark}

\newtheorem{algorithm}[theorem]{Algorithm}

\usepackage{epsfig}
\usepackage[usenames,dvipsnames]{pstricks}
\usepackage{pstricks-add}
\usepackage{pst-grad}
\usepackage{pst-plot} 
\usepackage{pst-node} 
\usepackage{pst-eucl} 
\usepackage{pst-coil} 
\usepackage[extdef=true]{delimset}
\DeclareMathDelimiterSet{\scal}[2]{
\selectdelim[l]<{#1}
\mathpunct{}\selectdelim[p]|
{#2}\selectdelim[r]>}
\DeclareMathDelimiterSet{\EC}[2]{
\mathsf{E}\selectdelim[l]({#1}
\mathpunct{}\selectdelim[p]|
{#2}\selectdelim[r])}

\newcommand{\menge}[2]{\bigl\{{#1}\mid{#2}\bigr\}} 
\DeclareMathDelimiterSet{\Menge}[2]{\selectdelim[l]\{
{#1}\selectdelim[m]|{#2}\selectdelim[r]\}}

\makeatletter
\def\upintkern@{\mkern-7mu\mathchoice{\mkern-3.5mu}{}{}{}}
\def\upintdots@{\mathchoice{\mkern-4mu\@cdots\mkern-4mu}%
{{\cdotp}\mkern1.5mu{\cdotp}\mkern1.5mu{\cdotp}}%
{{\cdotp}\mkern1mu{\cdotp}\mkern1mu{\cdotp}}%
{{\cdotp}\mkern1mu{\cdotp}\mkern1mu{\cdotp}}}
\makeatother
\DeclareFontFamily{OMX}{mdbch}{}
\DeclareFontShape{OMX}{mdbch}{m}{n}{ <->s * [0.8]  mdbchr7v }{}
\DeclareFontShape{OMX}{mdbch}{b}{n}{ <->s * [0.8]  mdbchb7v }{}
\DeclareFontShape{OMX}{mdbch}{bx}{n}{<->ssub * mdbch/b/n}{}
\DeclareSymbolFont{uplargesymbols}{OMX}{mdbch}{m}{n}
\SetSymbolFont{uplargesymbols}{bold}{OMX}{mdbch}{b}{n}
\DeclareMathSymbol{\upintop}{\mathop}{uplargesymbols}{82}
\DeclareMathSymbol{\upointop}{\mathop}{uplargesymbols}{"48}
\makeatletter
\renewcommand{\int}{\DOTSI\upintop\ilimits@}
\renewcommand{\oint}{\DOTSI\upointop\ilimits@}
\makeatother

\newcommand{\RR}{\mathbb{R}}
\newcommand{\NN}{\mathbb{N}}
\newcommand{\XX}{\EuScript{X}}

\newcommand{\CS}{\mathsf{C}}
\newcommand{\QS}{\mathsf{Q}}
\newcommand{\HS}{\mathsf{H}}
\newcommand{\GS}{\mathsf{G}}
\newcommand{\ZS}{\mathsf{Z}}
\newcommand{\zS}{\mathsf{z}}
\newcommand{\nS}{{\mathsf{n}}}
\newcommand{\nnn}{\mathsf{n}\in\mathbb{N}}
\newcommand{\jjj}{\mathsf{j}\in\mathbb{N}}
\newcommand{\iS}{\mathsf{i}}

\newcommand{\jS}{\mathsf{j}}
\newcommand{\kS}{\mathsf{k}}
\newcommand{\KS}{\mathsf{K}}
\newcommand{\xS}{\mathsf{x}}
\newcommand{\yS}{\mathsf{y}}

\newcommand{\TS}{\mathsf{T}}

\newcommand{\BE}{\EuScript{B}}
\newcommand{\FE}{\EuScript{F}}

\newcommand{\pinf}{{+}\infty}
\newcommand{\minf}{{-}\infty}
\newcommand{\zeroun}{\intv[o]{0}{1}}

\newcommand{\RPP}{\intv[o]0{0}{\pinf}}

\newcommand{\emp}{\varnothing}

\newcommand{\WC}{\ensuremath{{\mathfrak W}}}

\newcommand{\Sum}{\displaystyle\sum}

\newcommand{\minimize}[2]{\underset{\substack{{#1}}}
{\operatorname{minimize}}\;\;#2}

\newcommand{\pushfwd}%
{\ensuremath{\mbox{\Large$\,\triangleright\,$}}}

\DeclareMathOperator{\Argmin}{Argmin}
\DeclareMathOperator{\argmin}{argmin}
\newcommand{\Id}{\mathsf{Id}}

\DeclareMathOperator{\Fix}{Fix}

\DeclareMathOperator{\prox}{prox}
\DeclareMathOperator{\proj}{proj}

\newcommand{\EE}{\mathsf{E}}
\newcommand{\PP}{\mathsf{P}}

\renewcommand{\leq}{\leqslant}
\renewcommand{\geq}{\geqslant}
\newcommand{\exi}{\exists\,}
\newcommand{\weakly}{\rightharpoonup}

\newcommand{\Pas}{\text{\normalfont$\PP$-a.s.}}

\renewenvironment{abstract}{%
\vspace*{-0.50cm}
\small
\quotation%
\noindent%
{\normalfont\bfseries\sffamily
\nobreak\abstractname\ }%
}{%
\endquotation%
\medskip
}
\renewcommand{\abstractname}{Abstract.}
\newcommand\keywordsname{Keywords.}
\newenvironment{keywords}
{\renewcommand\abstractname{\keywordsname}\begin{abstract}}
{\end{abstract}}
\usepackage[auth-sc]{authblk}
\newcommand{\email}[1]{\href{mailto:#1}{\nolinkurl{#1}}}
\renewcommand*\Affilfont{\normalfont\normalsize}
\newcommand\affilcr{\protect\\ \protect\Affilfont}
\makeatletter
\renewcommand\AB@affilsepx{\protect\\[0.5em]}
\makeatother

\author[1]{Javier I. Madariaga}
\affil[1]{North Carolina State University
\affilcr
Department of Mathematics
\affilcr
Raleigh, NC 27695, USA
\affilcr
\email{jimadari@ncsu.edu}
}

\begin{document}

\title{An Abstract Stochastic Haugazeau Method\\ 
for Best Approximation\thanks{Contact author:
J. I. Madariaga. Email: \email{jimadari@ncsu.edu}.
This work was supported by the National
Science Foundation under grant DMS-2513409.
}}

\date{~}

\maketitle

\begin{abstract} 
The Haugazeau method was originally designed to compute the best
approximation from an intersection of closed convex sets in Hilbert
spaces using the projection operators onto the individual sets
iteratively. We propose an abstract stochastic version of it to
compute the best approximation from a closed convex set by
successive projections onto randomly generated stochastic outer
approximations of that set. Strong convergence in the mean square
and the almost sure modes is derived under general hypotheses on
the outer approximations. The results are applied to the
development of stochastic algorithms to construct the best 
approximation from an arbitrary intersection of fixed point sets by
random activation of blocks of operators. A numerical application
to the computation of Chebyshev centers is provided.
\end{abstract}


\newpage

\section{Introduction}
\label{sec:1}

Throughout this paper, $\HS$ is a separable real Hilbert space with
identity operator $\Id$, scalar product 
$\scal{\cdot}{\cdot}_{\HS}$, and associated norm $\|\cdot\|_{\HS}$.
$(\upOmega,\FE,\PP)$ is a complete probability space.

In his unpublished thesis, Yves Haugazeau presented a geometric
strategy for finding the projection onto the intersection $\ZS$ of
a finite collection $(\ZS_{\kS})_{1\leq\kS\leq\mathsf{p}}$
of closed convex subsets of $\HS$ by periodic
projections onto these sets individually. To describe it,
let $(\xS,\yS,\zS)\in\HS^3$ and define
\begin{equation}
\label{e:Q1}
\begin{cases}
\HS(\xS,\yS)&=\menge{\zS\in\HS}
{\scal{\zS-\yS}{\xS-\yS}_{\HS}\leq0};\\[2mm]
\mathsf{O}(\xS,\yS,\zS)&=
\begin{cases}
\HS(\xS,\yS)\cap\HS(\yS,\zS),&\text{if}\;\;
\HS(\xS,\yS)\cap\HS(\yS,\zS)\neq\emp;\\
\{\yS\},&\text{if}\;\;\HS(\xS,\yS)\cap\HS(\yS,\zS)=\emp;
\end{cases}\\[2mm]
\QS(\xS,\yS,\zS)&=\proj_{\mathsf{O}(\xS,\yS,\zS)}^{}\xS.
\end{cases}
\end{equation}
The half-space $\HS(\xS,\yS)$ is defined so that the
projection of $\xS$ onto $\HS(\xS,\yS)$ coincides with $\yS$, that
is, $\yS=\proj_{\HS(\xS,\yS)}^{}\xS$. Additionally,
$\mathsf{O}(\xS,\yS,\zS)$ is a
nonempty closed convex subset of $\HS$, which implies that
$\QS(\xS,\yS,\zS)$ is well-defined. Given a starting point
$\xS_{\mathsf{0}}\in\HS$, it is shown in
\cite[Th\'eor\`eme~3-1]{Haug68} that the sequence 
$(\xS_{\nS})_{\nnn}$ generated by the iterative algorithm 
\begin{equation}
\label{e:Hn}
(\forall\nnn)\quad
\xS_{\nS+1}=
\QS\brk1{\xS_{\mathsf{0}},\xS_{\nS},
\proj_{\ZS_{\nS\brk{\mathrm{mod}\,\mathsf{p}}+1}}\xS_{\nS}}
\end{equation}
converges strongly to the best approximation
$\proj_{\ZS}\xS_{\mathsf{0}}$ to $\xS_{\mathsf{0}}$ from $\ZS$. In
this process, $\xS_{\mathsf{0}}$ is projected onto the set
$\mathsf{O}\brk1{\xS_{\mathsf{0}},\xS_{\nS},
\proj_{\ZS_{\nS\brk{\mathrm{mod}\,\mathsf{p}}+1}}\xS_{\nS}}$, 
which is an outer approximation to $\ZS$. These ideas have also 
been used to design parallel projection methods to project
$\xS_{\mathsf{0}}$ onto $\ZS$ \cite{Pier76}. Note that
\begin{equation}
\ZS=\bigcap_{1\leq\kS\leq\mathsf{p}}\ZS_{\kS}\subset
\ZS_{\nS\brk{\mathrm{mod}\,\mathsf{p}}+1}\subset
\HS\brk1{\xS_{\nS},\mathsf{a}_{\nS}},\;\;\text{where}\;\;
\mathsf{a}_{\nS}
=\proj_{\ZS_{\nS\brk{\mathrm{mod}\,\mathsf{p}}+1}}
\xS_{\nS}.
\end{equation}
This observation led to the development in \cite{sicon1} of the
following abstract version of Haugazeau's method to find the best
approximation to $\xS_{\mathsf{0}}\in\HS$ from a nonempty closed
convex subset $\ZS$ of $\HS$. 
\begin{equation}
\label{e:t1-}
\hskip -1mm 
\begin{array}{l}
\textup{for}\;\nS=0,1,\ldots\\
\left\lfloor
\begin{array}{l}
\textup{take}\;\mathsf{a}_{\nS}\in\HS\;
\textup{such that}\;
\ZS\subset\HS(\xS_{\nS},\mathsf{a}_{\nS})\\
\textup{take}\;
\uplambda_{\nS}\in\left]0,1\right]\\
\mathsf{r}_{\nS}=\xS_{\nS}+\uplambda_{\nS}
\brk{\mathsf{a}_{\nS}-\xS_{\nS}}\\
\xS_{\nS+1}=\QS(\xS_{\mathsf{0}},\xS_{\nS},\mathsf{r}_{\nS}).
\end{array}
\right.\\
\end{array}
\end{equation}
This algorithm covers the original Haugazeau method, which is
obtained by setting
$\ZS=\bigcap_{1\leq\kS\leq\mathsf{p}}\ZS_{\kS}$,
and, for every $\nnn$, $\uplambda_{\nS}=1$ and
$\mathsf{a}_{\nS}=
\proj_{\ZS_{\nS\brk{\mathrm{mod}\,\mathsf{p}}+1}}
\xS_{\nS}$. In the general case, the strong convergence of the 
sequence $(\xS_{\nS})_{\nnn}$ in \eqref{e:t1-} to the solution of 
the best approximation problem is guaranteed as long as each weak
cluster point of $(\xS_{\nS})_{\nnn}$ belongs to $\ZS$
\cite{Mor1,sicon1}. Applications of this framework can be found
for instance in
\cite{Nfao15,Sadd22,sicon1,Acnu24,MaPr18,Marq25,Marq26,Solo00}.

In optimization theory, many stochastic methods have been proposed
to handle large-scale problems and high-dimensional settings; see,
e.g., \cite{Siim25,Siop15,Dieu23,Herm19,John24,Luke26} and their
bibliographies for discussions of the modeling and computational
benefits of stochastic methods. However, it remains an open
question whether stochastic versions of \eqref{e:t1-} can be
developed and if so, with which convergence properties. Closely
related to the best approximation problem is the convex feasibility
problem, which consists of finding an arbitrary point in the
intersection of a collection of closed convex sets. Stochastic
methods have been proposed for this problem in which, at each
iteration, only a finite randomly selected subcollection of sets is
activated \cite{Moco25,Herm19,Kost23,Neco21}. This feature is
especially valuable when the collection of sets is uncountably
infinite since, in that setting, deterministic methods cannot
guarantee convergence of the iterates by activating finite blocks
of sets. In this spirit, we propose to study the convergence of a
stochastic counterpart to \eqref{e:t1-} for finding the best
approximation to $x_{\mathsf{0}}\in L^2(\upOmega,\FE,\PP;\HS)$ from
an arbitrary collection of sets. This abstract stochastic Haugazeau
method operates as follows.

\begin{algorithm}
\label{algo:1}
Let $\ZS$ be a nonempty closed convex subset of $\HS$ and let 
$x_{\mathsf{0}}\in~L^2(\upOmega,\FE,\PP;\HS)$.
Iterate
\begin{equation}
\label{e:t1}
\hskip -1mm 
\begin{array}{l}
\textup{for}\;\nS=0,1,\ldots\\
\left\lfloor
\begin{array}{l}
\textup{take}\;a_{\nS}\in L^2(\upOmega,\FE,\PP;\HS)\;
\textup{such that}\;
\ZS\subset \HS(x_{\nS},a_{\nS})\;\Pas\\
\textup{take}\;
\lambda_{\nS}\in L^\infty(\upOmega,\FE,\PP;\left]0,1\right])\\
r_{\nS}=x_{\nS}+\lambda_{\nS}\brk{a_{\nS}-x_{\nS}}\\
x_{\nS+1}=\QS(x_{\mathsf{0}},x_{\nS},r_{\nS}).
\end{array}
\right.\\
\end{array}
\end{equation}
\end{algorithm}

In contrast to the deterministic setting, where the weak-to-strong
convergence principle discovered in \cite{Mor1} enables the
transformation of a broad range of weakly convergent methods in
nonlinear analysis into strongly convergent best approximation
methods \cite{Acnu24}, there is no stochastic weak-to-strong
convergence principle for stochastic methods. This is because they
are constructed over random outer approximations that are not
deterministic and do not act as exact cuts in the sense that $\ZS$
is not almost surely contained in the random outer approximation,
as illustrated in \cite[Figure~1]{Moco25}. Thus, a dedicated
analysis of Algorithm~\ref{algo:1} is required to establish 
convergence guarantees in this stochastic framework.

The paper is organized as follows. In Section~\ref{sec:2}, we
introduce the notation and preliminary results. The study of
Algorithm~\ref{algo:1} is presented in Section~\ref{sec:3}, where
we show the strong convergence of the sequences of iterates almost
surely and in $L^2(\upOmega,\FE,\PP;\HS)$. In Section~\ref{sec:4},
we apply the method to develop a randomly activated block-iterative
algorithm for solving a best approximation problem in which $\ZS$ is
described as the common fixed points of an arbitrary family of
operators, possibly uncountably infinite. In addition, we compare
Haugazeau's original cyclic method with a specialized version of
our method for finding the projection onto the finite intersection
of closed convex sets. Section~\ref{sec:5} concludes the paper with
a numerical application to the computation of Chebyshev centers of
nonempty and bounded sets in $\RR^{\mathsf{N}}$.
 
\section{Notation and background}
\label{sec:2}

\subsection{Notation}
\label{sec:21}

Random variables are denoted by italicized serif letters and
deterministic variables by sans-serif letters.

The symbols $\weakly$ and $\to$ denote weak and strong convergence
in $\HS$, respectively. The set of weak sequential cluster points
of a sequence $(\mathsf{x}_{\nS})_{\nnn}$ in $\HS$ is denoted by
$\WC(\mathsf{x}_{\nS})_{\nnn}$. The projection onto a nonempty
closed convex set $\CS\subset\HS$ is denoted by $\proj_{\CS}$. The
fixed point set of an operator $\TS\colon\HS\to\HS$ is
$\Fix\TS=\menge{\mathsf{x}\in\HS}{\TS\mathsf{x}=\mathsf{x}}$, $\TS$
is firmly quasinonexpansive \cite[Definition~4.1(iv)]{Livre1} if
\begin{equation}
(\forall\xS\in\HS)(\forall\mathsf{y}\in\Fix\TS)\quad
\norm{\TS\xS-\mathsf{y}}_{\HS}^2+\norm{\TS\xS-\xS}_{\HS}^2
\leq\norm{\xS-\mathsf{y}}_{\HS}^2, \end{equation} and $\TS$ is
demiclosed at $\yS\in\HS$ if for every $\xS\in\HS$ and every
sequence $(\xS_{\nS})_{\nnn}$ in $\HS$ such that
$\xS_{\nS}\weakly\xS$ and $\TS\xS_{\nS}\to\yS$, we have
$\TS\xS=\yS$. The reader is referred to \cite{Livre1} for
background on convex analysis and fixed point theory.
 
Let $(\upXi,\EuScript{G})$ be a measurable space. We say that
$x\colon(\upOmega,\FE,\PP)\to(\upXi,\EuScript{G})$ is a
$\upXi$-valued random variable (random variable for short) if it is
measurable. In particular, an $\HS$-valued random variable is a
measurable mapping $x\colon(\upOmega,\FE,\PP)\to(\HS,\BE_{\HS})$,
where $\BE_{\HS}$ denotes the Borel $\upsigma$-algebra of $\HS$.
The sub $\upsigma$-algebra of $\FE$ generated by a family $\upPhi$
of random variables is denoted by $\upsigma(\upPhi)$. Given
$x\colon\upOmega\to\upXi$ and $\mathsf{S}\in\EuScript{G}$, we set
$[x\in\mathsf{S}]=\menge{\upomega\in\upOmega}
{x(\upomega)\in\mathsf{S}}$. Let $x$ and $y$ be $\upXi$-valued
random variables. We say that $y$ is a copy of $x$ if, for every
$\mathsf{S}\in\EuScript{G}$,
$\PP([x\in\mathsf{S}])=\PP([y\in\mathsf{S}])$. Let $\XX$ be a sub
$\upsigma$-algebra of $\FE$. Then $L^2(\upOmega,\XX,\PP;\HS)$
denotes the space of equivalence classes of $\Pas$ equal
$\HS$-valued random variables
$x\colon(\upOmega,\XX,\PP)\to(\HS,\BE_{\HS})$ such that
$\EE\norm{x}_{\HS}^2<\pinf$. Endowed with the scalar product
\begin{equation}
\scal{\cdot}{\cdot}_{L^2(\upOmega,\XX,\PP;\HS)}\colon (x,y)\mapsto
\EE\scal{x}{y}_{\HS}=\int_{\upOmega}
\scal1{x(\upomega)}{y(\upomega)}_{\HS}\PP(d\upomega),
\end{equation} $L^2(\upOmega,\XX,\PP;\HS)$ is a real Hilbert space.
The reader is referred to \cite{Hyto16,Ledo91} for
background on probability in Hilbert spaces. 

\subsection{Preliminary results}
\label{sec:22}

A fundamental question regarding Algorithm~\ref{algo:1} is whether
$x_{\nS+1}$ is well-defined and measurable for every $\nnn$. To
provide an affirmative answer, it is necessary to first investigate
the properties of $\QS$.

\begin{lemma}[{\cite[Definition~2.8]{Livre1}}]
\label{l:Q2}
Set $\upchi=\scal{\mathsf{x-y}}{\mathsf{y-z}}_{\HS}$, 
$\upmu=\|\mathsf{x-y}\|_{\HS}^2$, $\upnu=\|\mathsf{y-z}\|_{\HS}^2$,
and $\uprho=\upmu\upnu-\upchi^2$. Then
\begin{equation}
\hskip -1mm \QS\left(\mathsf{x},\mathsf{y},
\mathsf{z}\right)=
\begin{cases}
\mathsf{y},&\text{if}\;\uprho=0\;\text{and}\;
\upchi< 0;\\[+0mm]
\mathsf{z},&\text{if}\;\uprho=0\;\text{and}\;
\upchi\geq 0;\\[+0mm]
\displaystyle
\mathsf{x}+(1+\upchi/\upnu)(\mathsf{z-y}), 
&\text{if}\;\uprho>0\;\text{and}\;
\upchi\upnu\geq\uprho;\\
\displaystyle \mathsf{y}+(\upnu/\uprho)
\big(\upchi(\mathsf{x-y})+\upmu(\mathsf{z-y})\big), 
&\text{if}\;\uprho>0\;\text{and}\;\upchi\upnu<\uprho.
\end{cases}
\end{equation}
\end{lemma}

\begin{lemma}
\label{l:13}
Let $\{x,y,z\}\subset L^2(\upOmega,\FE,\PP;\HS)$.
Suppose that there exists some $u\in L^2(\upOmega,\FE,\PP;\HS)$ 
such that $u\in\mathsf{O}(x,y,z)\;\Pas$
Then $\QS(x,y,z)\in L^2(\upOmega,\FE,\PP;\HS)$. In
particular, $\QS(x,y,z)\in L^2(\upOmega,\FE,\PP;\HS)$ if 
there exists $\mathsf{u}\in\HS$ such that
$\mathsf{u}\in\mathsf{O}(x,y,z)\;\Pas$
\end{lemma}
\begin{proof}
Set $\chi=\scal{x-y}{y-z}_{\HS}$, 
$\mu=\|x-y\|_{\HS}^2$, $\nu=\|y-z\|_{\HS}^2$, and 
$\rho=\mu\nu-\chi^2$. The continuity of the addition, the scalar 
multiplication, the scalar product on $\HS$, and the norm on $\HS$,
along with the measurability of $x$, $y$, and $z$, assure us that
$\chi$, $\mu$, $\nu$, and $\rho$ are measurable. 
Define the disjoint measurable sets
\begin{equation}
\label{e:111}
\begin{cases}
\mathsf{S}_1=\brk[s]1{\rho=0\;\text{and}\;\chi<0};&\;
\mathsf{S}_2=\brk[s]1{\rho=0\;\text{and}\;\chi\geq0};\\
\mathsf{S}_3=\brk[s]1{\rho>0\;\text{and}\;\chi\nu\geq\rho};&\;
\mathsf{S}_4=\brk[s]1{\rho>0\;\text{and}\;\chi\nu<\rho}.
\end{cases}
\end{equation}
Then it follows from Lemma~\ref{l:Q2} and \eqref{e:111} that we can
write $\QS(x,y,z)$ as
\begin{multline}
\QS(x,y,z)=
\mathsf{1}_{\mathsf{S}_1}y+\mathsf{1}_{\mathsf{S}_2}z
+\mathsf{1}_{\mathsf{S}_3}
\brk2{x+\brk2{1+\frac{\chi}{\nu+\mathsf{1}_{[\nu=0]}}}(z-y)}\\
+\mathsf{1}_{\mathsf{S}_4}
\brk2{y+\brk2{\frac{\nu}{\rho+\mathsf{1}_{[\rho=0]}}
\brk1{\chi(x-y)+\mu(z-y)}}}\;\;\Pas,
\end{multline}
which shows that $\QS(x,y,z)$ is measurable, as 
$(\upOmega,\FE,\PP)$ is complete. On the other hand, we deduce from
\eqref{e:Q1} and the assumptions that
\begin{equation}
\frac{1}{2}\EE\norm1{\QS(x,y,z)-x}_{\HS}^2
=\frac{1}{2}\EE\norm1{\proj_{\mathsf{O}(x,y,z)}^{}x-x}_{\HS}^2
\leq\frac{1}{2}\EE\norm{u-x}_{\HS}^2
\leq\EE\norm{u}_{\HS}^2+\EE\norm{x}_{\HS}^2.
\end{equation}
Hence, since $\{x,u\}\subset L^2(\upOmega,\FE,\PP;\HS)$,
we conclude that $\QS(x,y,z)\in L^2(\upOmega,\FE,\PP;\HS)$.
\end{proof}

\begin{lemma}[{\cite[Lemma~2.8]{Moco25}}]
\label{l:7}
Let $\boldsymbol{x}=(x_1,\ldots,x_{\mathsf{N}})$ be an
$\HS^\mathsf{N}$-valued random variable,
let $(\mathsf{K},\EuScript{K})$ be a measurable space, and suppose
that the random variable
$k\colon(\upOmega,\FE)\to(\mathsf{K},\EuScript{K})$ is
independent of $\upsigma(\boldsymbol{x})$. Let 
$\mathsf{f}\colon(\mathsf{K}\times\HS,\EuScript{K}\otimes\BE_{\HS})
\to\RR$ be measurable and such 
that $\EE\abs{\mathsf{f}(k,x_1)}<\pinf$, and define 
$\mathsf{g}\colon\HS\to\RR\colon\xS\mapsto\EE\mathsf{f}(k,\xS)$.
Then, for $\PP$-almost every $\upomega'\in\upOmega$,
\begin{equation}
\EC1{\mathsf{f}(k,x_1)}{\upsigma(\boldsymbol{x})}(\upomega')
=\int_{\upOmega}\mathsf{f}\brk1{k(\upomega),x_1(\upomega')}
\PP(d\upomega)
=\mathsf{g}\brk1{x_1(\upomega')}.
\end{equation}
\end{lemma}

\section{Convergence analysis}
\label{sec:3}

In this section, we establish the strong convergence of the
sequence $(x_{\nS})_{\nnn}$ generated by Algorithm~\ref{algo:1} to
the solution of the best approximation problem, in both the almost
sure and $L^2(\upOmega,\FE,\PP;\HS)$ modes.

\begin{theorem}
\label{t:1}
Let $(x_{\nS})_{\nnn}$ be the sequence generated by
Algorithm~\ref{algo:1}.
Then the following hold:
\begin{enumerate}
\item
\label{t:1i-}
$(x_{\nS})_{\nnn}$ is a well-defined sequence in
$L^2(\upOmega,\FE,\PP;\HS)$.
\item
\label{t:1i}
$(\forall\nnn)\;\ZS\subset \HS(x_{\mathsf{0}},x_{\nS})\cap 
\HS(x_{\nS},r_{\nS})\;\Pas$
\item
\label{t:1ii}
$(\exi\ell\in L^2(\upOmega,\FE,\PP;\RR))\;
\norm{x_{\nS}-x_{\mathsf{0}}}_{\HS}\uparrow\ell
\leq\norm{\proj_{\ZS}x_{\mathsf{0}}-x_{\mathsf{0}}}_{\HS}\;\Pas$ 
\item
\label{t:1ii-}
$(\norm{x_{\nS}}_{\HS})_{\nnn}$ is bounded $\Pas$
\item
\label{t:1iii}
$\sum_{\nnn}\EE\norm{x_{\nS+1}-x_{\nS}}_{\HS}^2<\pinf$ and
$\sum_{\nnn}\EC{\norm{x_{\nS+1}-x_{\nS}}_{\HS}^2}{\XX_{\nS}}
<\pinf\;\Pas$
\item
\label{t:1iv}
$\sum_{\nnn}\EE\brk{\lambda_{\nS}^2\norm{a_{\nS}-x_{\nS}}_{\HS}^2}
<\pinf$ and 
$\sum_{\nnn}\EC{\lambda_{\nS}^2\norm{a_{\nS}-x_{\nS}}_{\HS}^2}
{\XX_{\nS}}<\pinf\;\Pas$
\item
\label{t:1v}
Suppose that $\mathfrak{W}(x_{\nS})_{\nnn}\subset\ZS\;\Pas$
Then $(x_{\nS})_{\nnn}$ converges strongly $\Pas$ and strongly in
$L^2(\upOmega,\FE,\PP;\HS)$ to $\proj_{\ZS}x_{\mathsf{0}}$.
\end{enumerate}
\end{theorem}

\begin{proof}
\ref{t:1i-} and \ref{t:1i}: Suppose that, for some $\nnn$, 
$x_{\nS}\in L^2(\upOmega,\FE,\PP;\HS)$ and
$\ZS\subset\HS(x_{\mathsf{0}},x_{\nS})\;\Pas$ It is 
clear from \eqref{e:t1} that $r_{\nS}\in 
L^2(\upOmega,\FE,\PP;\HS)$. On the other hand, for $\PP$-almost 
every $\upomega\in\upOmega$,
\begin{align}
&\HS\brk1{(x_{\nS}(\upomega),a_{\nS}(\upomega)}\nonumber\\
&\qquad=\Menge1{\zS\in\HS}
{\scal{\zS-a_{\nS}(\upomega)}
{x_{\nS}(\upomega)-a_{\nS}(\upomega)}_{\HS}
\leq 0}\nonumber\\
&\qquad=\Menge1{\zS\in\HS}
{\scal{\zS-a_{\nS}(\upomega)}
{x_{\nS}(\upomega)-x_{\nS}(\upomega)
-\lambda_{\nS}\brk{a_{\nS}(\upomega)-x_{\nS}(\upomega)}}_{\HS}
\leq0}\nonumber\\
&\qquad=\Menge1{\zS\in\HS}
{\scal{\zS-r_{\nS}(\upomega)}
{x_{\nS}(\upomega)-r_{\nS}(\upomega)}_{\HS}
\leq\scal{a_{\nS}(\upomega)-r_{\nS}(\upomega)}
{x_{\nS}(\upomega)-r_{\nS}(\upomega)}_{\HS}}\nonumber\\
&\qquad=\Menge1{\zS\in\HS}
{\scal{\zS-r_{\nS}(\upomega)}
{x_{\nS}(\upomega)-r_{\nS}(\upomega)}_{\HS}
\leq-\lambda_{\nS}\brk{1-\lambda_{\nS}}
\norm{x_{\nS}(\upomega)-a_{\nS}(\upomega)}_{\HS}^2}\nonumber\\
&\qquad\subset\HS\brk1{x_{\nS}(\upomega),r_{\nS}(\upomega)}.
\label{e:t12}
\end{align}
We deduce from \eqref{e:t1} and \eqref{e:t12} that
$\ZS\subset\HS(x_{\nS},r_{\nS})\;\Pas$ Hence
there exists a subset $\upOmega_{\nS}'\in\FE$ such that 
$\PP(\upOmega_{\nS}')=1$ and
$(\forall\upomega\in\upOmega_{\nS}')\;\;
\ZS\subset\HS(x_{\nS}(\upomega),r_{\nS}(\upomega))$. 
Likewise, there exists $\upOmega_{\nS}''\in\FE$ such that 
$\PP(\upOmega_{\nS}')=1$ and
$(\forall\upomega\in\upOmega_{\nS}'')\;
\ZS\subset\HS(x_{\mathsf{0}}(\upomega),x_{\nS}(\upomega))$. Let
$\upomega\in\upOmega_{\nS}'\cap\upOmega_{\nS}''$. It follows from 
\cite[Theorem~3.16]{Livre1} and \eqref{e:t1} that
\begin{align}
\ZS\subset\HS\brk1{x_{\mathsf{0}}(\upomega),x_{\nS}(\upomega)}\cap
\HS\brk1{x_{\nS}(\upomega),r_{\nS}(\upomega)}\;\Rightarrow\;
&\ZS\subset\HS\brk1{x_{\mathsf{0}}(\upomega),
\QS
\brk{x_{\mathsf{0}}(\upomega),x_{\nS}(\upomega),r_{\nS}(\upomega)}
}\nonumber\\
\;\Leftrightarrow\;
&\ZS\subset\HS\brk{x_{\mathsf{0}}(\upomega),x_{\nS+1}(\upomega)}.
\label{e:712}
\end{align}
Since $\PP(\upOmega_{\nS}'\cap\upOmega_{\nS}'')=1$, we get 
$\emp\neq\ZS\subset\HS(x_{\mathsf{0}},x_{\nS+1})\;\Pas$ Hence, 
Lemma~\ref{l:13} yields $x_{\nS+1}\in 
L^2(\upOmega,\FE,\PP;\HS)$. Since $x_{\mathsf{0}}\in 
L^2(\upOmega,\FE,\PP;\HS)$ and 
$\ZS\subset\HS(x_{\mathsf{0}},x_{\mathsf{0}})=\HS\;\Pas$,
we conclude by an inductive argument that $(x_{\nS})_{\nnn}$ and
$(r_{\nS})_{\nnn}$ are sequences in $L^2(\upOmega,\FE,\PP;\HS)$, 
and, for every $\nnn$, $\ZS\subset \HS(x_{\mathsf{0}},x_{\nS})\cap 
\HS(x_{\nS},r_{\nS})\;\Pas$

\ref{t:1ii}: Set, for every $\nnn$,
$\XX_{\nS}=\upsigma(x_{\mathsf{0}},\ldots,x_{\nS})$, and consider 
the random process 
$(\norm{x_{\nS}-x_{\mathsf{0}}}_{\HS},\XX_{\nS})_{\nnn}$. 
Let us show that this process is a submartingale. Let $\nnn$. It
follows from \ref{t:1i} that
\begin{equation}
\label{e:112}
\mathsf{O}(x_{\mathsf{0}},x_{\nS},r_{\nS})
=\HS(x_{\mathsf{0}},x_{\nS})\cap\HS(x_{\nS},r_{\nS})\;\Pas\;\;
\text{and}\;\;x_{\nS+1}\in
\HS(x_{\mathsf{0}},x_{\nS})\cap\HS(x_{\nS},r_{\nS})\;\Pas
\end{equation}
We note from \eqref{e:Q1} that
$x_{\nS}=\proj_{\HS(x_{\mathsf{0}},x_{\nS})}^{}x_{\mathsf{0}}\;
\Pas$, and from \eqref{e:112} that 
$x_{\nS+1}\in\HS(x_{\mathsf{0}},x_{\nS})\;\Pas$ Then 
\begin{equation}
(\forall\nnn)\quad \norm{x_{\nS}-x_{\mathsf{0}}}_{\HS} 
\leq\norm{x_{\nS+1}-x_{\mathsf{0}}}_{\HS}\quad\Pas
\end{equation}
Hence
\begin{equation}
(\forall\nnn)\quad\norm{x_{\nS}-x_{\mathsf{0}}}_{\HS}=
\EC1{\norm{x_{\nS}-x_{\mathsf{0}}}_{\HS}}{\XX_{\nS}}
\leq\EC1{\norm{x_{\nS+1}-x_{\mathsf{0}}}_{\HS}}{\XX_{\nS}}\quad\Pas
\end{equation}
Therefore
$(\norm{x_{\nS}-x_{\mathsf{0}}}_{\HS},\XX_{\nS})_{\nnn}$ is a 
positive submartingale. On the other hand, for every $\nnn$,
$\proj_{\ZS}x_{\mathsf{0}}\in
\ZS\subset\HS(x_{\mathsf{0}},x_{\nS})\;\Pas$ 
Then, for every $\nnn$, $\norm{x_{\nS}-x_{\mathsf{0}}}_{\HS}\leq
\norm{\proj_{\ZS}x_{\mathsf{0}}-x_{\mathsf{0}}}_{\HS}\;\Pas$, which
shows that
\begin{equation}
(\forall\nnn)\quad\EE\norm{x_{\nS}-x_{\mathsf{0}}}_{\HS}^2
\leq\EE\norm{\proj_{\ZS}x_{\mathsf{0}}-x_{\mathsf{0}}}_{\HS}^2<
\pinf.
\end{equation}
Consequently, 
$\sup_{\nnn}\EE\norm{x_{\nS}-x_{\mathsf{0}}}_{\HS}^2<\pinf$ and we
deduce from \cite[\S IV~Theorems~4.1s(i) and 4.1s(iii)]{Doob53}
that $(\norm{x_{\nS}-x_{\mathsf{0}}}_{\HS})_{\nnn}$ converges 
$\Pas$ and converges in $L^2(\upOmega,\FE,\PP;\RR)$ to a random 
variable $\ell\in L^2(\upOmega,\FE,\PP;\RR)$ which satisfies 
$\ell\leq
\norm{\proj_{\ZS}x_{\mathsf{0}}-x_{\mathsf{0}}}_{\HS}\;\Pas$ 

\ref{t:1ii-}: Note that, for every $\nnn$,
\begin{equation}
\label{e:boun}
\norm{x_{\nS}}_{\HS}
\leq\norm{x_{\nS}-x_{\mathsf{0}}}_{\HS}+\norm{x_{\mathsf{0}}}_{\HS}
\;\;\Pas
\end{equation}
We deduce from \eqref{e:boun} and \ref{t:1ii} that
$(\norm{x_{\nS}}_{\HS})_{\nnn}$ is bounded $\Pas$

\ref{t:1iii}: Let $\nnn$. Since 
$x_{\nS+1}\in\HS(x_{\mathsf{0}},x_{\nS})\;\Pas$, we have 
\begin{align}
\norm{x_{\nS+1}-x_{\nS}}_{\HS}^2 
&\leq\norm{x_{\nS+1}-x_{\nS}}_{\HS}^2 
+2\scal{x_{\nS+1}-x_{\nS}}{x_{\nS}-x_{\mathsf{0}}}_{\HS}\nonumber\\
&=\norm{x_{\nS+1}-x_{\mathsf{0}}}_{\HS}^2
-\norm{x_{\nS}-x_{\mathsf{0}}}_{\HS}^2\;\;\Pas
\end{align}
Hence, $\sum_{\jS=0}^{\nS}\norm{x_{\jS+1}-x_{\jS}}_{\HS}^2
\leq\norm{x_{\nS+1}-x_{\mathsf{0}}}_{\HS}^2
\leq\norm{\proj_{\ZS}x_{\mathsf{0}}-x_{\mathsf{0}}}_{\HS}^2\;\Pas$ 
Therefore, by taking expected value and limit $\nS\to\pinf$, we 
have 
\begin{equation}
\sum_{\jjj}
\EE\brk2{\EC1{\norm{x_{\jS+1}-x_{\jS}}_{\HS}^2}{\XX_{\jS}}}
=\sum_{\jjj}\EE\norm{x_{\jS+1}-x_{\jS}}_{\HS}^2
\leq\EE\norm{\proj_{\ZS}x_{\mathsf{0}}-x_{\mathsf{0}}}_{\HS}^2
<\pinf.
\end{equation}
We conclude that 
$\sum_{\jjj}\EE\norm{x_{\jS+1}-x_{\jS}}_{\HS}^2<\pinf$ and that
$\sum_{\jjj}\EC{\norm{x_{\jS+1}-x_{\jS}}_{\HS}^2}{\XX_{\jS}}
<\pinf\;\Pas$

\ref{t:1iv}: Let $\nnn$. Since
$x_{\nS+1}\in\HS(x_{\nS},r_{\nS})\;\Pas$, we have
\begin{align}
\norm{r_{\nS}-x_{\nS}}_{\HS}^2 
&\leq\norm{x_{\nS+1}-r_{\nS}}_{\HS}^2
+\norm{x_{\nS}-r_{\nS}}_{\HS}^2\nonumber\\
&\leq\norm{x_{\nS+1}-r_{\nS}}_{\HS}^2
+2\scal{x_{\nS+1}-r_{\nS}}{r_{\nS}-x_{\nS}}_{\HS} 
+\norm{x_{\nS}-r_{\nS}}_{\HS}^2\nonumber\\
&=\norm{x_{\nS+1}-x_{\nS}}_{\HS}^2\;\;\Pas
\end{align}
It follows then from \ref{t:1iii} and \eqref{e:t1} that
\begin{equation}
\sum_{\jjj}\EE\brk2{
\EC{\lambda_{\jS}^2\norm{a_{\jS}-x_{\jS}}_{\HS}^2}{\XX_{\jS}}}=
\sum_{\jjj}\EE\brk1{\lambda_{\jS}^2\norm{a_{\jS}-x_{\jS}}_{\HS}^2}=
\sum_{\jjj}\EE\norm{r_{\jS}-x_{\jS}}_{\HS}^2<\pinf.
\end{equation}
Hence, $\sum_{\jjj}
\EE\brk{\lambda_{\jS}^2\norm{a_{\jS}-x_{\jS}}_{\HS}^2}<\pinf$ and
$\sum_{\jjj}
\EC{\lambda_{\jS}^2\norm{a_{\jS}-x_{\jS}}_{\HS}^2}{\XX_{\jS}}
<\pinf\;\Pas$

\ref{t:1v}: It follows from \ref{t:1ii-} that
$\mathfrak{W}(x_{\nS})_{\nnn}\neq\emp\;\Pas$ Now suppose that
$\mathfrak{W}(x_{\nS})_{\nnn}\subset\ZS\;\Pas$ Then there exists
$\upOmega'\in\FE$ such that 
$\PP(\upOmega')=1$ and $(\forall\upomega\in\upOmega')\;
\emp\neq\mathfrak{W}(x_{\nS}(\upomega))_{\nnn}\subset\ZS$. In 
addition, consider from \ref{t:1ii} the event $\upOmega''\in\FE$ 
such that $\PP(\upOmega'')=1$ and
\begin{equation}
\label{e:ell1}
(\forall\upomega\in\upOmega'')
\quad
\begin{cases}
\norm{x_{\nS}(\upomega)-x_{\mathsf{0}}(\upomega)}_{\HS}\to
\ell(\upomega);\\[1mm]
\ell(\upomega)\leq
\norm{\proj_{\ZS}x_{\mathsf{0}}(\upomega)
-x_{\mathsf{0}}(\upomega)}_{\HS}.
\end{cases}
\end{equation}
Let $\upomega\in\upOmega'\cap\upOmega''$ and let $\xS\in\ZS$ be 
a weak sequential cluster point of $(x_{\nS}(\upomega))_{\nnn}$, 
say $x_{\kS_{\nS}}(\upomega)\weakly\xS$. Then, the weak 
lower semicontinuity of $\|\cdot\|_{\HS}$ and \eqref{e:ell1}
imply that
\begin{align}
\norm{\xS-x_{\mathsf{0}}(\upomega)}_{\HS}
&\leq\varliminf_{\nnn}
\norm{x_{\kS_{\nS}}(\upomega)-x_{\mathsf{0}}(\upomega)}_{\HS}
\nonumber\\
&=\ell(\upomega)\nonumber\\
&\leq\norm{\proj_{\ZS}x_{\mathsf{0}}(\upomega)
-x_{\mathsf{0}}(\upomega)}_{\HS}\nonumber\\
&=\inf_{\yS\in\ZS}\norm{\yS-x_{\mathsf{0}}(\upomega)}_{\HS}
\nonumber\\
&\leq\norm{\xS-x_{\mathsf{0}}(\upomega)}_{\HS}.
\label{e:t6}
\end{align}
Therefore $\xS=\proj_{\ZS}x_{\mathsf{0}}(\upomega)$ is the only
weak sequential cluster point of the sequence 
$(x_{\nS}(\upomega))_{\nnn}$ and 
$x_{\nS}(\upomega)\weakly\proj_{\ZS}x_{\mathsf{0}}(\upomega)$. 
Hence, $x_{\nS}(\upomega)-x_{\mathsf{0}}(\upomega)\weakly
\proj_{\ZS}x_{\mathsf{0}}(\upomega)-x_{\mathsf{0}}(\upomega)$ and 
\eqref{e:t6} yields
$\norm{x_{\nS}(\upomega)-x_{\mathsf{0}}(\upomega)}_{\HS}\to
\norm{\proj_{\ZS}x_{\mathsf{0}}(\upomega)
-x_{\mathsf{0}}(\upomega)}_{\HS}$. Then, by 
\cite[Lemma~2.41(i)]{Livre1}, we get that
$x_{\nS}(\upomega)-x_{\mathsf{0}}(\upomega)\to
\proj_{\ZS}x_{\mathsf{0}}(\upomega)-x_{\mathsf{0}}(\upomega)$ and, 
therefore,
$x_{\nS}(\upomega)\to\proj_{\ZS}x_{\mathsf{0}}(\upomega)$. Since 
$\PP(\upOmega'\cap\upOmega'')=1$, we deduce that 
$x_{\nS}\to\proj_{\ZS}x_{\mathsf{0}}\;\Pas$ Furthermore, note from
\ref{t:1ii} that, for every $\nnn$,
\begin{align}
\norm{\proj_{\ZS}x_{\mathsf{0}}-x_{\nS}}_{\HS}
&\leq\norm{\proj_{\ZS}x_{\mathsf{0}}-x_{\mathsf{0}}}_{\HS}
+\norm{x_{\nS}-x_{\mathsf{0}}}_{\HS}\nonumber\\
&\leq2\norm{\proj_{\ZS}x_{\mathsf{0}}-x_{\mathsf{0}}}_{\HS}\;\;\Pas
\end{align}
Since $ \norm{\proj_{\ZS}x_{\mathsf{0}}-x_{\mathsf{0}}}_{\HS}\in
L^2(\upOmega,\FE,\PP;\RR)$, the dominated convergence theorem
guarantees that $(x_{\nS})_{\nnn}$ converges strongly in
$L^2(\upOmega,\FE,\PP;\HS)$ to $\proj_{\ZS}x_{\mathsf{0}}$.
\end{proof}

\section{Application to common fixed point best approximation}
\label{sec:4}

In this section we specialize the best approximation problem to the
following setting.

\begin{problem}
\label{prob:1}
Let $(\mathsf{K},\EuScript{K})$ be a measurable space and
let $(\TS_{\kS})_{\kS\in\mathsf{K}}$ be a family of firmly
quasinonexpansive operators such that 
$\boldsymbol{\TS}\colon
(\mathsf{K}\times\HS,\EuScript{K}\otimes\BE_{\HS})
\to(\HS,\BE_{\HS})\colon(\kS,\xS)
\mapsto\TS_{\kS}\xS$ is measurable and, for every 
$\kS\in\mathsf{K}$, $\Id-\TS_{\kS}$ is demiclosed at
$\mathsf{0}$. Let $k$ be a $\mathsf{K}$-valued random variable,
suppose that
\begin{equation}
\label{e:Z}
\ZS=\Menge1{\mathsf{z}\in\HS}{\mathsf{z}\in\Fix\TS_{k}\;\Pas}
\neq\emp,
\end{equation} 
and let $x_{\mathsf{0}}\in L^2(\upOmega,\FE,\PP;\HS)$. The goal is 
to find $\proj_{\ZS}x_{\mathsf{0}}$.
\end{problem}

To emphasize the versatility of Problem~\ref{prob:1}, 
we present some examples of firmly quasinonexpansive operators
commonly used in nonlinear analysis and optimization methods. 

\begin{example}[{\cite[Proposition~2.3]{Mor1}}]
\label{ex:51}
Let $\TS\colon\HS\to\HS$. Then $\TS$ is firmly quasinonexpansive
if one of the following holds:
\begin{enumerate}
\item
\label{ex:51i}
$\ZS$ is a nonempty closed convex subset of $\HS$ and
$\TS=\proj_{\ZS}$. Here, $\Fix\TS=\ZS$.
\item
\label{ex:51ii}
$\mathsf{f}\colon\HS\to\left]\minf,\pinf\right]$ is a proper
lower semicontinuous convex function and
\begin{equation}
\TS=\prox_{\mathsf{f}}\colon\HS\to\HS\colon\xS\mapsto
\underset{\mathsf{y}\in\HS}{\argmin}
\brk2{\mathsf{f}(\mathsf{y})+
\frac{1}{2}\norm{\xS-\mathsf{y}}_{\HS}^2}.
\end{equation}
Here, $\Fix\TS=\Argmin\mathsf{f}$.
\item
\label{ex:51iii}
$\mathsf{A}\colon\HS\to 2^{\HS}$ is maximally monotone and
$\TS=\mathsf{J}_{\mathsf{A}}=(\Id+\mathsf{A})^{-1}$.
Here, $\Fix\TS=\menge{\zS\in\HS}{\mathsf{0}\in\mathsf{A}\zS}$.
\item
\label{ex:51iv}
$\mathsf{f}\colon\HS\to\RR$ is a continuous convex function, 
$\mathsf{s}\colon\HS\to\HS\colon\xS\mapsto\mathsf{s}(\xS)\in
\partial\mathsf{f}(\xS)$ is a selection of $\partial\mathsf{f}$, 
and
\begin{equation}
\nonumber
\TS=\GS_{\mathsf{f}}\colon\HS\to\HS\colon\mathsf{x}\mapsto
\begin{cases}
\xS-\dfrac{\mathsf{f}(\xS)}{\norm{\mathsf{s}(\xS)}_{\HS}^2}
\mathsf{s}(\xS),&\text{if}\;\;\mathsf{f}(\xS)>0;\\
\xS,&\text{if}\;\;\mathsf{f}(\xS)\leq 0,
\end{cases}
\end{equation}
is the subgradient projector onto
$\Fix\TS=\menge{\xS\in\HS}{\mathsf{f}(\xS)\leq0}$.
\end{enumerate} 
\end{example}

\subsection{Algorithm and convergence}
\label{sec:41}

Since the set $\ZS$ of Problem~\ref{prob:1} is a nonempty closed
convex subset of $\HS$ \cite[Remark~5.3]{Moco25}, we invoke
Theorem~\ref{t:1} to introduce the first randomly selected
block-iterative fixed point algorithm to solve
Problem~\ref{prob:1}, which guarantees both strong almost
sure convergence and strong convergence in
$L^2(\upOmega,\FE,\PP;\HS)$. We emphasize that this method is novel
even for $\mathsf{K}$ finite or countably infinite. Furthermore,
unlike the deterministic algorithms, this stochastic method is able
to solve Problem~\ref{prob:1} in its generality by random
activation of block of operators.

\begin{theorem}
\label{t:2}
In the setting of Problem~\ref{prob:1}, let $0<\mathsf{M}\in\NN$,
$\updelta\in\left]0,1/\mathsf{M}\right[$, and
$\upvarepsilon\in\zeroun$. Iterate
\begin{equation}
\label{e:p2}
\begin{array}{l}
\textup{for}\;\nS=0,1,\ldots\\
\left\lfloor
\begin{array}{l}
\XX_{\nS}=\upsigma(x_{\mathsf{0}},\dots,x_{\nS})\\
\begin{array}{l}
\text{for}\;\iS=1,\dots,\mathsf{M}\\
\left\lfloor
\begin{array}{l}
k_{\iS,\nS}\;\textup{is a copy of $k$ and is independent of}\;
\XX_{\nS}\\
p_{\iS,\nS}=\TS_{k_{\iS,\nS}}x_{\nS}\\[1mm]
\end{array}
\right.\\
\end{array}\\[5mm]
(\beta_{\iS,\nS})_{1\leq\iS\leq\mathsf{M}}\;\textup{are}\;
\left[0,1\right]\textup{-valued random variables such that}\\
\qquad\sum_{\iS=1}^\mathsf{M}\beta_{\iS,\nS}=1\;\Pas
\;\textup{and}\;\;
(\forall\iS\in\{1,\dots,\mathsf{M}\})\;
\beta_{\iS,\nS}\geq\updelta
\mathsf{1}_{\bigl[\norm{p_{\iS,\nS}-x_{\nS}}_{\HS}=
\max\limits_{1\leq\jS\leq\mathsf{M}}
\norm{p_{\jS,\nS}-x_{\nS}}_{\HS}\bigr]}\;\\
p_{\nS}=\sum_{\iS=1}^\mathsf{M}\beta_{\iS,\nS}p_{\iS,\nS}\\[1mm]
L_{\nS}=\dfrac{{\sum_{\iS=1}^{\mathsf{M}}\beta_{\iS,\nS}
\norm{p_{\iS,\nS}-x_{\nS}}_{\HS}^2}+
\mathsf{1}_{\bigl[p_{\nS}=x_{\nS}\bigr]}}
{\norm{p_{\nS}-x_{\nS}}_{\HS}^2+
\mathsf{1}_{\bigl[p_{\nS}=x_{\nS}\bigr]}}\\[5mm]
a_{\nS}=x_{\nS}+L_{\nS}(p_{\nS}-x_{\nS})\\
\textup{take}\;\lambda_{\nS}\in 
L^\infty(\upOmega,\FE,\PP;\left[\upvarepsilon,1\right])\\
r_{\nS}=x_{\nS}+\lambda_{\nS}\brk{a_{\nS}-x_{\nS}}\\
x_{\nS+1}=\QS(x_{\mathsf{0}},x_{\nS},r_{\nS}).
\end{array}
\right.\\
\end{array}
\end{equation}
Then $(x_{\nS})_{\nnn}$ converges strongly $\Pas$ and strongly in 
$L^2(\upOmega,\FE,\PP;\HS)$ to $\proj_{\ZS}x_{\mathsf{0}}$.
\end{theorem}
\begin{proof}
We will show that the sequence constructed by \eqref{e:p2}
corresponds to a sequence generated by Algorithm~\ref{algo:1}. 
First, let us show that
\begin{equation}
\label{e:405}
(\forall\nnn)\quad L_{\nS}\geq1\;\:\Pas
\end{equation}
Fix $\zS\in\ZS$ and $\nnn$. For every 
$\iS\in\{1,\ldots,\mathsf{M}\}$, let
$\upOmega_{\iS,\nS}\in\FE$ be such that 
\begin{equation}
\label{e:s7}
\PP(\upOmega_{\iS,\nS})=1\quad\text{and}\quad
(\forall\upomega\in\upOmega_{\iS,\nS})\quad
\zS\in\Fix\TS_{k_{\iS,\nS}(\upomega)}. 
\end{equation}
Thanks to \eqref{e:s7}, we then choose
$\upOmega_{\nS}\in\FE$ such that 
\begin{equation}
\PP(\upOmega_{\nS})=1\quad\text{and}\quad
(\forall\upomega\in\upOmega_{\nS})\quad
\bigcap_{1\leq\iS\leq\mathsf{M}}\Fix\TS_{k_{\iS,\nS}(\upomega)}
\neq\emp\quad\text{and}\quad
\sum_{\iS=1}^\mathsf{M}\beta_{\iS,\nS}(\upomega)=1.
\end{equation}
Given $\upomega\in\upOmega_{\nS}$, we study the following two
cases:
\begin{itemize}
\item
Suppose that $p_{\nS}(\upomega)=x_{\nS}(\upomega)$. Then
\cite[Proposition~2.4]{Else01} shows that
\begin{equation}
x_{\nS}(\upomega)\in\Fix\brk3{\sum_{\iS=1}^{\mathsf{M}}
\beta_{\iS,\nS}(\upomega)\TS_{k_{\iS,\nS}(\upomega)}}=
\bigcap_{1\leq\iS\leq\mathsf{M}}\Fix\TS_{k_{\iS,\nS}(\upomega)},
\end{equation}
hence, for every $\iS\in\{1,\ldots,\mathsf{M}\}$, 
$x_{\nS}(\upomega)=p_{\iS,\nS}(\upomega)$. Thus,
\begin{equation}
L_{\nS}(\upomega)=
\dfrac{{\Sum_{\iS=1}^{\mathsf{M}}\beta_{\iS,\nS}(\upomega)
\norm1{p_{\iS,\nS}(\upomega)-x_{\nS}(\upomega)}_{\HS}^2}+
\mathsf{1}_{\bigl[p_{\nS}=x_{\nS}\bigr]}(\upomega)}
{\norm1{p_{\nS}(\upomega)-x_{\nS}(\upomega)}_{\HS}^2+
\mathsf{1}_{\bigl[p_{\nS}=x_{\nS}\bigr]}(\upomega)}
=\dfrac{\mathsf{1}_{\bigl[p_{\nS}=x_{\nS}\bigr]}(\upomega)}
{\mathsf{1}_{\bigl[p_{\nS}=x_{\nS}\bigr]}(\upomega)}=1.
\end{equation}
\item
Suppose that
$p_{\nS}(\upomega)\neq x_{\nS}(\upomega)$. Then the convexity of
$\|\cdot\|_{\HS}^2$ yields
\begin{equation}
0<\norm1{p_{\nS}(\upomega)-x_{\nS}(\upomega)}_{\HS}^2
=\norm3{\sum_{\iS=1}^{\mathsf{M}}
\beta_{\iS,\nS}(\upomega)
\brk1{p_{\iS,\nS}(\upomega)-x_{\nS}(\upomega)}}_{\HS}^2
\leq\sum_{\iS=1}^{\mathsf{M}}\beta_{\iS,\nS}(\upomega)
\norm1{p_{\iS,\nS}(\upomega)-x_{\nS}(\upomega)}_{\HS}^2,
\end{equation}
which implies that $L_{\nS}(\upomega)\geq1$. 
\end{itemize}
Next, we will show by induction that
$(x_{\nS})_{\nnn}$ and $(a_{\nS})_{\nnn}$ are in 
$L^2(\upOmega,\FE,\PP;\HS)$. Let 
$\iS\in\{1,\ldots,\mathsf{M}\}$ and suppose that 
$x_{\nS}\in L^2(\upOmega,\FE,\PP;\HS)$. Then
$\TS_{k_{\iS,\nS}}x_{\nS}=
\boldsymbol{\TS}\circ(k_{\iS,\nS},x_{\nS})$ is measurable.
Furthermore, for every $\upomega\in\upOmega_{\iS,\nS}$, 
$2\TS_{k_{\iS,\nS}(\upomega)}-\Id$ is quasinonexpansive 
with $\Fix(2\TS_{k_{\iS,\nS}(\upomega)}-\Id)=
\Fix\TS_{k_{\iS,\nS}(\upomega)}$
\cite[Proposition~2.2(v)]{Else01}. Hence
\begin{align}
2\norm1{p_{\iS,\nS}(\upomega)}^2_{\HS}
&=\dfrac{1}{2}
\norm1{2\TS_{k_{\iS,\nS}(\upomega)}x_{\nS}(\upomega)}^2_{\HS}
\nonumber\\
&\leq\norm1{\brk1{2\TS_{k_{\iS,\nS}(\upomega)}
-\Id}x_{\nS}(\upomega)-\zS}^2_{\HS}
+\norm1{x_{\nS}(\upomega)+\zS}^2_{\HS}
\nonumber\\
&\leq\norm1{x_{\nS}(\upomega)-\zS}^2_{\HS}
+\norm1{x_{\nS}(\upomega)+\zS}^2_{\HS}.
\end{align}
Thus, since 
$x_{\nS}\in L^2(\upOmega,\FE,\PP;\HS)$ and $\zS\in\HS$, we have
$p_{\iS,\nS}\in L^2(\upOmega,\FE,\PP;\HS)$. Hence, \eqref{e:p2} 
yields that $p_{\nS}\in L^2(\upOmega,\FE,\PP;\HS)$ and that 
$a_{\nS}$ is measurable. To show that $a_{\nS}\in
L^2(\upOmega,\FE,\PP;\HS)$, we first note that
\begin{align}
x_{\nS}-a_{\nS}
&=L_{\nS}\brk1{x_{\nS}-p_{\nS}}\nonumber\\
&=\dfrac{{\sum_{\iS=1}^{\mathsf{M}}\beta_{\iS,\nS}
\|p_{\iS,\nS}-x_{\nS}\|_{\HS}^2}+
\mathsf{1}_{\bigl[p_{\nS}=x_{\nS}\bigr]}}
{\|p_{\nS}-x_{\nS}\|_{\HS}^2+
\mathsf{1}_{\bigl[p_{\nS}=x_{\nS}\bigr]}}
\brk1{x_{\nS}-p_{\nS}}
\nonumber\\
&=\dfrac{{\sum_{\iS=1}^{\mathsf{M}}\beta_{\iS,\nS}
\brk1{\scal{x_{\nS}}{x_{\nS}-p_{\iS,\nS}}_{\HS}
-\scal{p_{\iS,\nS}}{x_{\nS}-p_{\iS,\nS}}_{\HS}}}}
{\|x_{\nS}-p_{\nS}\|_{\HS}^2+
\mathsf{1}_{\bigl[p_{\nS}=x_{\nS}\bigr]}}
\brk1{x_{\nS}-p_{\nS}}
\nonumber\\
&=\dfrac{\scal{x_{\nS}}{x_{\nS}-p_{\nS}}_{\HS}
-\sum_{\iS=1}^{\mathsf{M}}\beta_{\iS,\nS}
\scal{p_{\iS,\nS}}{x_{\nS}-p_{\iS,\nS}}_{\HS}}
{\|x_{\nS}-p_{\nS}\|_{\HS}^2+
\mathsf{1}_{\bigl[p_{\nS}=x_{\nS}\bigr]}}
\brk1{x_{\nS}-p_{\nS}}\label{e:998-}\\
&=\dfrac{\scal{x_{\nS}}{x_{\nS}-p_{\nS}}_{\HS}
-\sum_{\iS=1}^{\mathsf{M}}\beta_{\iS,\nS}
\scal{p_{\iS,\nS}-\zS}{x_{\nS}-p_{\iS,\nS}}_{\HS}
-\sum_{\iS=1}^{\mathsf{M}}\beta_{\iS,\nS}
\scal{\zS}{x_{\nS}-p_{\iS,\nS}}_{\HS}}
{\|x_{\nS}-p_{\nS}\|_{\HS}^2+
\mathsf{1}_{\bigl[p_{\nS}=x_{\nS}\bigr]}}
\brk1{x_{\nS}-p_{\nS}}\nonumber\\
&=\dfrac{\scal{x_{\nS}-\zS}{x_{\nS}-p_{\nS}}_{\HS}
-\sum_{\iS=1}^{\mathsf{M}}\beta_{\iS,\nS}
\scal{p_{\iS,\nS}-\zS}{x_{\nS}-p_{\iS,\nS}}_{\HS}}
{\|x_{\nS}-p_{\nS}\|_{\HS}^2+
\mathsf{1}_{\bigl[p_{\nS}=x_{\nS}\bigr]}}
\brk1{x_{\nS}-p_{\nS}}\;\;\Pas
\label{e:998}
\end{align}
On the other hand, it follows from 
\cite[Proposition~4.2(iv)]{Livre1} that
\begin{equation}
(\forall\iS\in\{1,\dots,\mathsf{M}\})\quad
\scal1{p_{\iS,\nS}-\zS}{x_{\nS}-p_{\iS,\nS}}_{\HS}
=\scal1{\TS_{k_{\iS,\nS}}x_{\nS}-\zS}
{x_{\nS}-\TS_{k_{\iS,\nS}}x_{\nS}}_{\HS}\geq0\;\;\Pas
\end{equation}
In turn, the concavity of
$\mathsf{y}\mapsto\scal{\mathsf{y}-\zS}
{x_{\nS}(\upomega)-\mathsf{y}}_{\HS}$ yields
\begin{equation}
\label{e:c2}
0\leq\sum_{\iS=1}^{\mathsf{M}}\beta_{\iS,\nS}
\scal1{p_{\iS,\nS}-\zS}{x_{\nS}-p_{\iS,\nS}}_{\HS}
\leq \scal1{p_{\nS}-\zS}{x_{\nS}-p_{\nS}}_{\HS}\;\;\Pas
\end{equation}
and therefore it follows from the convexity of the norm square 
and \eqref{e:998} that
\begin{align}
\dfrac{1}{2}\EE\|x_{\nS}-a_{\nS}\|_{\HS}^2
&=\dfrac{1}{2}\EE
\norm3{\dfrac{\scal{x_{\nS}-\zS}{x_{\nS}-p_{\nS}}_{\HS}
-\sum_{\iS=1}^{\mathsf{M}}\beta_{\iS,\nS}
\scal{p_{\iS,\nS}-\zS}{x_{\nS}-p_{\iS,\nS}}_{\HS}}
{\|x_{\nS}-p_{\nS}\|_{\HS}^2+
\mathsf{1}_{\bigl[p_{\nS}=x_{\nS}\bigr]}}
\brk1{x_{\nS}-p_{\nS}}}^2_{\HS}\nonumber\\
&=\dfrac{1}{2}\EE
\abs3{\dfrac{\scal{x_{\nS}-\zS}{x_{\nS}-p_{\nS}}_{\HS}
-\sum_{\iS=1}^{\mathsf{M}}\beta_{\iS,\nS}
\scal{p_{\iS,\nS}-\zS}{x_{\nS}-p_{\iS,\nS}}_{\HS}}
{\|x_{\nS}-p_{\nS}\|_{\HS}+
\mathsf{1}_{\bigl[p_{\nS}=x_{\nS}\bigr]}}}^2\nonumber\\
&\leq\EE
\abs3{\dfrac{\scal{x_{\nS}-\zS}{x_{\nS}-p_{\nS}}_{\HS}}
{\|x_{\nS}-p_{\nS}\|_{\HS}
+\mathsf{1}_{\bigl[p_{\nS}=x_{\nS}\bigr]}}}^2
+\EE\abs3{\dfrac{
\sum_{\iS=1}^{\mathsf{M}}\beta_{\iS,\nS}
\scal{p_{\iS,\nS}-\zS}{x_{\nS}-p_{\iS,\nS}}_{\HS}}
{\|x_{\nS}-p_{\nS}\|_{\HS}
+\mathsf{1}_{\bigl[p_{\nS}=x_{\nS}\bigr]}}}^2\nonumber\\
&\leq\EE
\abs3{\dfrac{\norm{x_{\nS}-\zS}_{\HS}\norm{x_{\nS}-p_{\nS}}_{\HS}}
{\|x_{\nS}-p_{\nS}\|_{\HS}
+\mathsf{1}_{\bigl[p_{\nS}=x_{\nS}\bigr]}}}^2
+\EE\abs3{\dfrac{
\scal{p_{\nS}-\zS}{x_{\nS}-p_{\nS}}_{\HS}}
{\|x_{\nS}-p_{\nS}\|_{\HS}
+\mathsf{1}_{\bigl[p_{\nS}=x_{\nS}\bigr]}}}^2\nonumber\\
&\leq\EE\norm{x_{\nS}-\zS}_{\HS}^2
+\EE\abs3{\dfrac{
\norm{p_{\nS}-\zS}_{\HS}\norm{x_{\nS}-p_{\nS}}_{\HS}}
{\|x_{\nS}-p_{\nS}\|_{\HS}
+\mathsf{1}_{\bigl[p_{\nS}=x_{\nS}\bigr]}}}^2\nonumber\\
&\leq\EE\norm{x_{\nS}-\zS}_{\HS}^2
+\EE\norm{p_{\nS}-\zS}_{\HS}^2\nonumber\\
&<\pinf.
\end{align}
Since $\{x_{\nS},p_{\nS}\}\subset L^2(\upOmega,\FE,\PP;\HS)$
and $\zS\in\HS$, we thus obtain
$x_{\nS}-a_{\nS}\in L^2(\upOmega,\FE,\PP;\HS)$ and hence
$a_{\nS}\in L^2(\upOmega,\FE,\PP;\HS)$. Moreover, it follows 
from 
\eqref{e:p2}
that $r_{\nS}\in L^2(\upOmega,\FE,\PP;\HS)$. On the other hand, we 
deduce from
\eqref{e:p2} and \eqref{e:998-} that
\begin{align}
\scal{a_{\nS}}{x_{\nS}-a_{\nS}}_{\HS}
&=L_{\nS}\scal1{x_{\nS}+L_{\nS}(p_{\nS}-x_{\nS})}
{x_{\nS}-p_{\nS}}_{\HS}\nonumber\\
&=L_{\nS}\brk2{\scal{x_{\nS}}{x_{\nS}-p_{\nS}}_{\HS}
-L_{\nS}\norm{p_{\nS}-x_{\nS}}_{\HS}^2}\nonumber\\
&=L_{\nS}\brk2{\scal{x_{\nS}}{x_{\nS}-p_{\nS}}_{\HS}
-\scal{x_{\nS}}{x_{\nS}-p_{\nS}}_{\HS}
+\sum_{\iS=1}^{\mathsf{M}}
\beta_{\iS,\nS}\scal{p_{\iS,\nS}}{x_{\nS}-p_{\iS,\nS}}_{\HS}}
\nonumber\\
&=L_{\nS}\sum_{\iS=1}^{\mathsf{M}}
\beta_{\iS,\nS}\scal{p_{\iS,\nS}}{x_{\nS}-p_{\iS,\nS}}_{\HS}
\;\;\Pas
\label{e:406}
\end{align}
Furthermore, \eqref{e:405}, \eqref{e:c2}, and \eqref{e:406} yield
\begin{align}
\scal{\zS}{x_{\nS}-a_{\nS}}_{\HS}
&=L_{\nS}\scal{\zS}{x_{\nS}-p_{\nS}}_{\HS}\nonumber\\
&=L_{\nS}\sum_{\iS=1}^{\mathsf{M}}\beta_{\iS,\nS}
\scal{\zS}{x_{\nS}-p_{\iS,\nS}}_{\HS}\nonumber\\
&\leq L_{\nS}\sum_{\iS=1}^{\mathsf{M}}\beta_{\iS,\nS}
\scal{p_{\iS,\nS}}{x_{\nS}-p_{\iS,\nS}}_{\HS}\nonumber\\
&=\scal{a_{\nS}}{x_{\nS}-a_{\nS}}_{\HS}\;\;\Pas,
\end{align}
which shows that $\ZS\subset\HS(x_{\nS},a_{\nS})\;\Pas$
We can repeat the arguments in \eqref{e:t12}--\eqref{e:712} to show
that $x_{\nS+1}\in L^2(\upOmega,\FE,\PP;\HS)$. Since 
$x_{\mathsf{0}}\in L^2(\upOmega,\FE,\PP;\HS)$, we conclude 
inductively that
\begin{equation}
(\forall\nnn)\quad \{x_{\nS},a_{\nS}\}\subset 
L^2(\upOmega,\FE,\PP;\HS)\;\;
\text{and}\;\;\ZS\subset\HS(x_{\nS},a_{\nS})\;\Pas
\end{equation}
Therefore, the sequence $(x_{\nS})_{\nnn}$ constructed by 
\eqref{e:p2} corresponds to one generated by 
Algorithm~\ref{algo:1}.
It is therefore enough to show that
$\mathfrak{W}(x_{\nS})_{\nnn}\subset\ZS\;\Pas$ to prove the claim. 
For this purpose, we infer first from \eqref{e:p2} that
\begin{align}
\label{e:21}
\EC1{\norm{a_{\nS}-x_{\nS}}_{\HS}^2}{\XX_{\nS}}
&=\EC2{\norm1{L_{\nS}\brk1{p_{\nS}-x_{\nS}}}_{\HS}^2}{\XX_{\nS}}
\nonumber\\
&=\EC2{\abs1{L_{\nS}}^2\norm1{p_{\nS}-x_{\nS}}_{\HS}^2}{\XX_{\nS}}
\nonumber\\
&=\EC3{L_{\nS}\sum_{\iS=1}^{\mathsf{M}}\beta_{\iS,\nS}
\norm1{p_{\iS,\nS}-x_{\nS}}_{\HS}^2}{\XX_{\nS}}\nonumber\\
&\geq\EC3{\sum_{\iS=1}^{\mathsf{M}}\beta_{\iS,\nS}
\norm1{p_{\iS,\nS}-x_{\nS}}_{\HS}^2}{\XX_{\nS}}\nonumber\\
&\geq\EC2{\updelta\max_{1\leq\jS\leq\mathsf{M}}
\norm1{p_{\jS,\nS}-x_{\nS}}_{\HS}^2}{\XX_{\nS}}\nonumber\\
&\geq\updelta\EC2{\norm1{p_{1,\nS}-x_{\nS}}_{\HS}^2}{\XX_{\nS}}
\nonumber\\
&=\updelta
\EC2{\norm1{\TS_{k_{1,\nS}}x_{\nS}-x_{\nS}}_{\HS}^2}{\XX_{\nS}}.
\end{align}
However, $k_{1,\nS}$ is independent of $\XX_{\nS}$ and is a copy of
$k$. Thus, Lemma~\ref{l:7}
guarantees that, for $\PP$-almost every $\upomega'\in\upOmega$, 
\begin{align}
\EC2{\norm1{\TS_{k_{1,\nS}}x_{\nS}-x_{\nS}}_{\HS}^2}{\XX_{\nS}}
(\upomega')
&=\int_{\upOmega}\norm1{\TS_{k_{1,\nS}(\upomega)}
x_{\nS}(\upomega')-x_{\nS}(\upomega')}_{\HS}^2\PP(d\upomega)
\nonumber\\
&=\int_{\upOmega}\norm1{\TS_{k(\upomega)}
x_{\nS}(\upomega')-x_{\nS}(\upomega')}_{\HS}^2\PP(d\upomega).
\end{align}
Therefore, for $\PP$-almost every $\upomega'\in\upOmega$, 
\eqref{e:21} yields 
\begin{equation}
\label{e:22}
\EC1{\lambda_{\nS}^2\norm{a_{\nS}-x_{\nS}}_{\HS}^2}{\XX_{\nS}}
(\upomega')
\geq\upvarepsilon^2\updelta\int_{\upOmega}\norm1{\TS_{k(\upomega)}
x_{\nS}(\upomega')-x_{\nS}(\upomega')}_{\HS}^2\PP(d\upomega)\;\;
\Pas
\end{equation}
Taking the expected value in \eqref{e:22}, summing over
$\nnn$, and applying Theorem~\ref{t:1}\ref{t:1iv}, we find that
\begin{equation}
\EE\brk3{\sum_{\nnn}\int_{\upOmega}
\norm1{\TS_{k(\upomega)}x_{\nS}-x_{\nS}}_{\HS}^2\PP(d\upomega)}
=\sum_{\nnn}\EE\brk3{\int_{\upOmega}
\norm1{\TS_{k(\upomega)}x_{\nS}-x_{\nS}}_{\HS}^2\PP(d\upomega)}
<\pinf,
\end{equation}
which implies
\begin{equation}
\sum_{\nnn}\int_{\upOmega}\norm1{
\TS_{k(\upomega)}x_{\nS}-x_{\nS}}_{\HS}^2
\PP(d\upomega)<\pinf\;\;\Pas
\end{equation} 
This leads to the existence of a set $\upOmega'\in\FE$ such that
$\PP(\upOmega')=1$ and, 
for every $\upomega'\in\upOmega'$, the series 
$\sum_{\nnn}\int_{\upOmega}
\norm{\TS_{k(\upomega)}x_{\nS}(\upomega')
-x_{\nS}(\upomega')}_{\HS}^2\PP(d\upomega)$ converges.
Furthermore, in view of Theorem~\ref{t:1}\ref{t:1ii-}, we can 
assume that $\WC(x_{\nS}(\upomega'))_{\nnn}\neq\emp$ on 
$\upOmega'$. Fix $\upomega'\in\upOmega'$ and let $x(\upomega')\in
\WC(x_{\nS}(\upomega'))_{\nnn}$, say 
$x_{\jS_{\nS}}(\upomega')\weakly x(\upomega')$. It follows from the
monotone convergence theorem that
\begin{equation}
\int_{\upOmega}\sum_{\nnn}
\norm1{\TS_{k(\upomega)}x_{\nS}(\upomega')
-x_{\nS}(\upomega')}_{\HS}^2\PP(d\upomega)
=\sum_{\nnn}\int_{\upOmega}
\norm1{\TS_{k(\upomega)}x_{\nS}(\upomega')
-x_{\nS}(\upomega')}_{\HS}^2\PP(d\upomega)<\pinf.
\end{equation}
Hence, for $\PP$-almost every $\upomega\in\upOmega$, 
$\sum_{\nnn}\norm{\TS_{k(\upomega)}x_{\nS}(\upomega')
-x_{\nS}(\upomega')}_{\HS}^2<\pinf$.
Therefore, there exists 
$\upOmega''\in\FE$ such that $\PP(\upOmega'')=1$ and 
\begin{equation}
\brk1{\forall\upomega\in\upOmega''}\quad 
\TS_{k(\upomega)}x_{\nS}(\upomega')-
x_{\nS}(\upomega')\to 0.
\end{equation}
It then follows from the demiclosedness of the operators 
$(\Id-\TS_{\kS})_{\kS\in\mathsf{K}}$ at $\mathsf{0}$ that
\begin{equation}
\brk1{\forall\upomega\in\upOmega''}\quad
\TS_{k(\upomega)}x(\upomega')=x(\upomega').
\end{equation}
Therefore $x(\upomega')\in\menge{\mathsf{z}\in\HS}{\mathsf{z}\in
\Fix\TS_{k}\;\Pas}=\ZS$. Since $\upomega'$ is arbitrarily taken in 
$\upOmega'$, we conclude that 
$\WC(x_{\nS})_{\nnn}\subset\ZS\;\Pas$,
and the claim thus follows from Theorem~\ref{t:1}\ref{t:1v}.
\end{proof}

\begin{remark}
The algorithm of Theorem~\ref{t:2} is inspired by the deterministic
extrapolated algorithm proposed in \cite[Section~6.5]{sicon1},
where the collection of sets is at most countably infinite and the
control rule $\kS_{\iS,\nS}$, the weights $\upbeta_{\iS,\nS}$, and
the relaxation parameters $\uplambda_{\nS}$ are deterministic.
\end{remark}

\begin{remark}
The random activation of indices in Theorem~\ref{t:2} replaces the
control rule strategies of deterministic methods. Thus, the same
approach to randomize the control rule can be used, for instance,
to design randomized block-iterative strongly convergent
algorithms to find the best approximation from the Kuhn-Tucker set
associated with a primal-dual monotone inclusion problem
\cite{Sadd22,Acnu24,MaPr18}. 
\end{remark}

\subsection{Best approximation from a finite collection of closed
convex sets}
\label{sec:42}

To connect Theorem~\ref{t:2} with Haugazeau's original work, we
specialize it to solve the following simple best approximation
problem.

\begin{problem}
\label{prob:2}
Let $(\ZS_{\kS})_{1\leq\kS\leq\mathsf{p}}$ be a finite
collection of closed convex subsets of $\HS$ such that
$\ZS=\bigcap_{1\leq\kS\leq\mathsf{p}}\ZS_{\kS}\neq\emp$ and
let $\xS_{\mathsf{0}}\in\HS$. The goal is to find
$\proj_{\ZS}\xS_{\mathsf{0}}$. 
\end{problem}

To implement Theorem~\ref{t:2}, we need first to embed
Problem~\ref{prob:2} into the setting of Problem~\ref{prob:1}. Let
us define then 
\begin{equation}
\label{e:kdef}
\begin{cases}
\KS=\{1,\ldots,\mathsf{p}\};\\
\EuScript{K}=2^{\KS};\\
k\;\text{is any}\;\KS\text{-valued random
variable such that}\;
(\forall\kS\in\{1,\ldots,\mathsf{p}\})\;\PP\brk1{[k=\kS]}>0.
\end{cases}
\end{equation}
Then $\ZS=\menge{\zS\in\HS}{\zS\in\ZS_{k}\;\Pas}
=\menge{\zS\in\HS}{\zS\in\Fix\proj_{\ZS_{k}}\;\Pas}$ where,
following Example~\ref{ex:51}\ref{ex:51i}, the firmly
quasinonexpansive operators correspond to the projectors. We remark
that any random variable $k$ satisfying \eqref{e:kdef} suffices to
define the randomized method. In particular, $k$ may be uniformly
distributed over $\KS$, i.e., for every
$\kS\in\{1,\ldots,\mathsf{p}\}$, $\PP([k=\kS])=1/\mathsf{p}$.
However, \eqref{e:kdef} also allows for the implementation of 
nonuniform weights.

We implement Theorem~\ref{t:2}
with $\lambda_{\nS}\equiv 1$ and $\mathsf{M}=1$, which
implies that $L_{\nS}\equiv 1$ and $r_{\nS}\equiv a_{\nS}\equiv
p_{\nS}\equiv p_{1,\nS}$.
\begin{corollary}
\label{c:2}
In the setting of Problem~\ref{prob:2} and \eqref{e:kdef}, iterate
\begin{equation}
\begin{array}{l}
\textup{for}\;\nS=0,1,\ldots\\
\left\lfloor
\begin{array}{l}
k_{\nS}\;\textup{is a copy of $k$ and is independent of}\;
\upsigma(\xS_{\mathsf{0}},\dots,x_{\nS})\\
x_{\nS+1}=\QS\brk1{\xS_{\mathsf{0}},x_{\nS},
\proj_{\ZS_{k_{\nS}}}x_{\nS}}.
\end{array}
\right.\\
\end{array}
\end{equation}
Then $(x_{\nS})_{\nnn}$ converges strongly $\Pas$ and strongly in 
$L^2(\upOmega,\FE,\PP;\HS)$ to $\proj_{\ZS}\xS_{\mathsf{0}}$.
\end{corollary}

\begin{remark}\
\label{r:1}
\begin{enumerate}
\item
\label{r:1i}
The proposed randomized strategy differs fundamentally from
Haugazeau's original cyclic approach. The cyclic method follows a
deterministic order, ensuring that every set is activated exactly
once every $\mathsf{p}$ iterations. This places a strict bound
on how long any specific set is ignored. On the other hand, the
random strategy selects the active index independently at each
iteration. While this approach does not enforce a fixed time
between visits to a specific set, it still ensures that all sets
are visited infinitely often with probability one. Hence,
it provides a simpler alternative that avoids rigid patterns.
Moreover, the expected waiting time to activate a specific index
$\kS$ remains finite since it is the inverse of the probability of
activation, e.g., it is equal to $\mathsf{p}$ in the case of
uniformly distributed $k$.
\item
\label{r:1ii}
Another way to analyze the asymptotic behavior of the algorithm is
to fix $\upomega\in\upOmega$ and study the behavior of
$(x_{\nS}(\upomega))_{\nnn}$ via the control rule
$(k_{\nS}(\upomega))_{\nnn}$. 
In this case, the convergence of $(x_{\nS}(\upomega))_{\nnn}$
cannot be established by appealing to deterministic results since
$(k_{\nS}(\upomega))_{\nnn}$ fails to satisfy the necessary 
assumptions; see \cite[Definition~3.1 and Theorem~3.1]{sicon1}. 
\item
\label{r:1iii}
To compare Haugazeau's original cyclic approach with the randomized
strategy, we use $\lambda_{\nS}\equiv 1$ and $\mathsf{M}=1$ in
Corollary~\ref{c:2}. However, randomly selected block-iterative
versions can also be implemented, where the extrapolation
$L_{\nS}\geq 1$ will significantly accelerate the convergence, as
shown for deterministic algorithms \cite{Sign03,jat4,Pier76}. 
\end{enumerate}
\end{remark}

\section{Application to the computation of a Chebyshev center}
\label{sec:5}

We specialize the setting to $\HS=\RR^{\mathsf{N}}$ and 
let $\mathsf{S}$ be a nonempty bounded subset of $\HS$. A Chebyshev
center of $\mathsf{S}$ and the Chebyshev radius of $\mathsf{S}$ are
a solution and the optimal value, respectively, of the problem
\begin{equation}
\label{e:chb}
\minimize{\xS\in\HS}\sup_{\yS\in\mathsf{S}}\;\norm{\xS-\yS}_{\HS}.
\end{equation}
We refer to \cite[Chapter~15]{Alim22} for further background on 
Chebyshev centers.
The goal is to find a Chebyshev center of $\mathsf{S}$ and the
Chebyshev radius of $\mathsf{S}$.
However, this task can be difficult, particularly
if the set $\mathsf{S}$ is nonconvex, which makes \eqref{e:chb}
nonconvex as well. Another way to rewrite this problem is as the
following constrained minimization in a higher space:
\begin{equation}
\label{e:chb2}
\minimize{(\xS,\uprho)\in\HS\times\RR}
\uprho\;\;\text{subject to}\;\;
(\xS,\uprho)\in\bigcap_{\yS\in\mathsf{S}}
\menge{(\zS,\upxi)\in\HS\times\RR}
{\norm{\zS-\yS}_{\HS}\leq\upxi}.
\end{equation}
For every $\yS\in\mathsf{S}$, the set
$\menge{(\zS,\upxi)\in\HS\times\RR}
{\norm{\zS-\yS}_{\HS}\leq\upxi}$ is closed
and convex, and so is the intersection. 
We propose the following
$\upalpha$-approximation of \eqref{e:chb2}.

\begin{problem}
\label{prob:3}
Set $\mathsf{K}=\mathsf{S}$, let $\EuScript{K}$ be the Borel
$\upsigma$-algebra of $\HS$ restricted to $\mathsf{S}$, and let $k$
be an $\mathsf{S}$-valued random variable with uniform distribution
over $\mathsf{S}$. Let $\upalpha\in\RPP$ and set
$(\xS_{\mathsf{0}},\uprho_{\mathsf{0}})=(\mathsf{0},-\upalpha)$.
The task is to
\begin{equation}
\minimize{(\xS,\uprho)\in\HS\times\RR}
\norm1{(\xS,\uprho)
-(\xS_{\mathsf{0}},\uprho_{\mathsf{0}})}_{\HS\times\RR}^2
\;\;\text{subject to}\;\;
(\xS,\uprho)\in
\menge{(\zS,\upxi)\in\HS\times\RR}
{\norm{\zS-k}_{\HS}\leq\upxi\;\;\Pas}.
\end{equation}
\end{problem}

\begin{figure}[b!]
\centering
\begin{tabular}{c@{}c@{}c@{}}
\includegraphics[width=5.0cm,height=5.0cm]{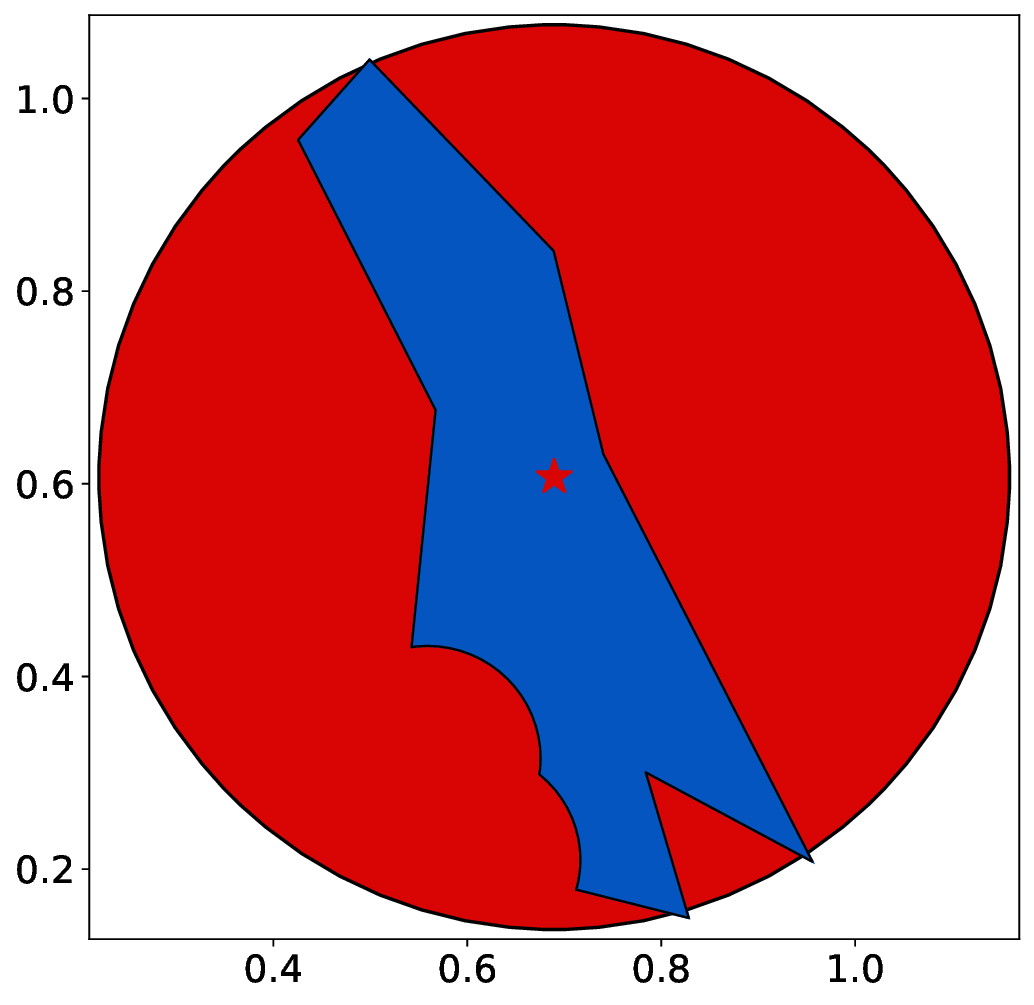}&
\hspace{0.1cm}
\includegraphics[width=5.0cm,height=5.0cm]{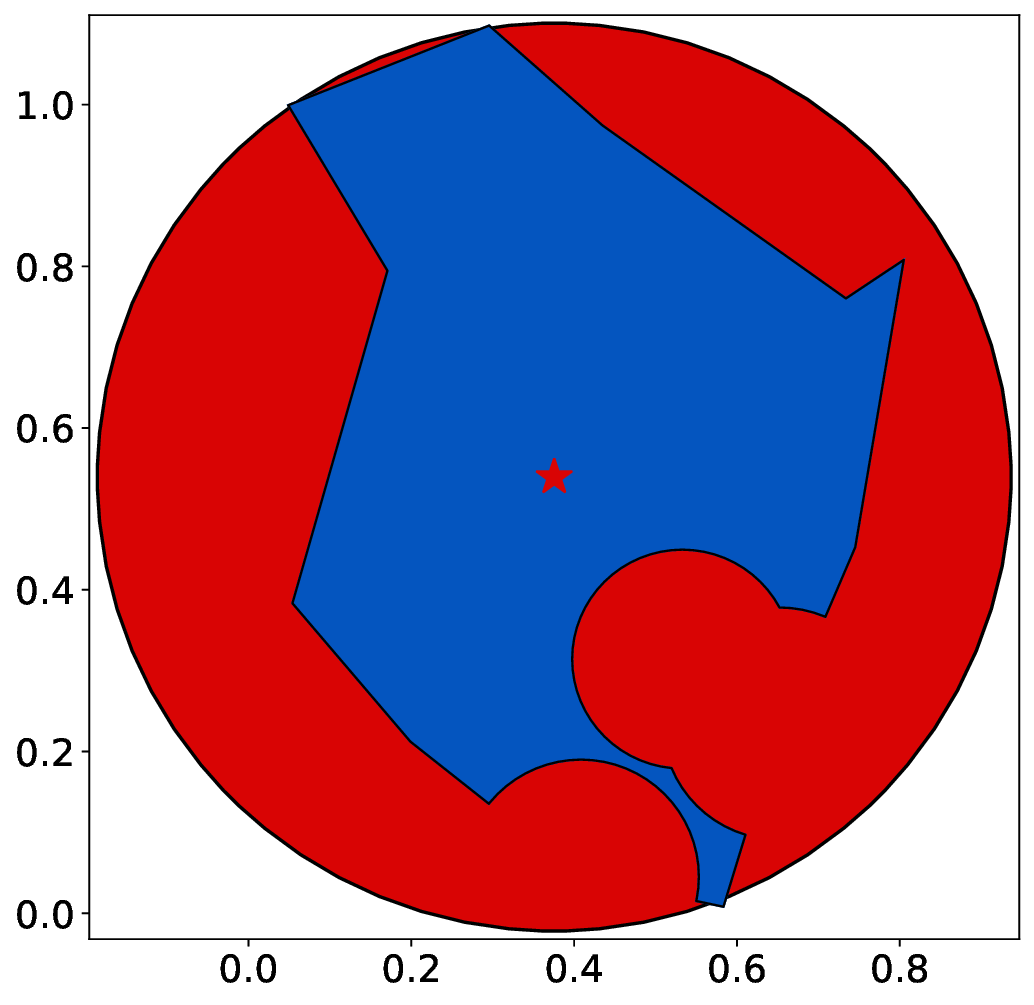}&
\hspace{0.1cm}
\includegraphics[width=5.0cm,height=5.0cm]{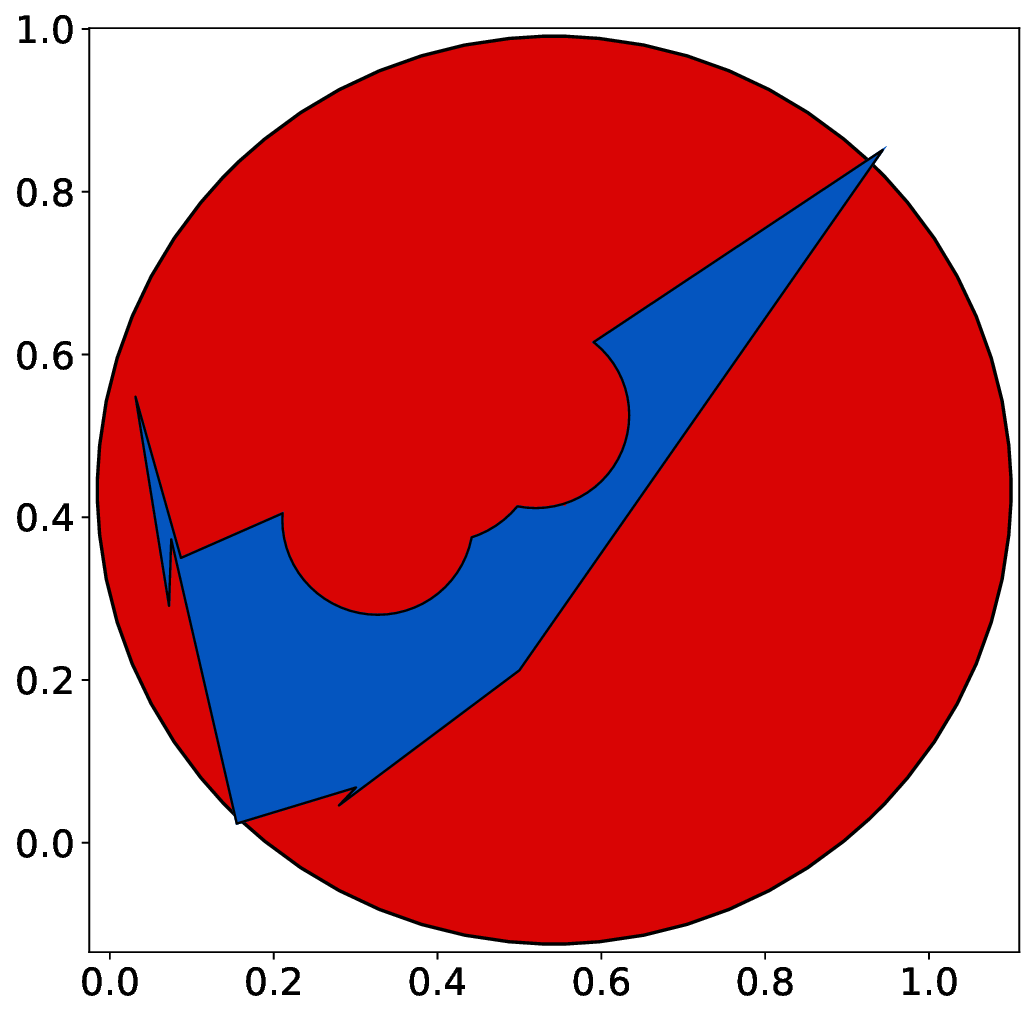}
\end{tabular} 
\caption{
Three instances of the experiment of Section~\ref{sec:5}.
{\color{dblue} Blue}: The subset $\mathsf{S}$. 
{\color{dred} Red}: The ball with center and radius given by the 
$\upalpha$-approximation.}
\label{fig:ex2.1}
\end{figure}%

This $\upalpha$-approximation of \eqref{e:chb2} is a best
approximation problem in the sense of Problem~\ref{prob:1} under
the quasinonexpansive operators defined by the projectors, as in
Example~\ref{ex:51}\ref{ex:51i}. The objective function
\begin{equation}
(\xS,\uprho)\mapsto\norm1{(\xS,\uprho)
-(\xS_{\mathsf{0}},\uprho_{\mathsf{0}})}_{\HS\times\RR}^2=
\norm{\xS}_{\HS}^2+\abs{\uprho}^2
+2\upalpha\uprho+\abs{\upalpha}^2
\end{equation}
corresponds to a regularization of the objective function of
\eqref{e:chb2} and the distance between their minimizers vanishes
as $\upalpha\uparrow\pinf$. Furthermore, we replace the strict
intersection over every $\yS\in\mathsf{S}$ with the almost sure
intersection defined by $k$. 

In our experiment, we solve three instances of Problem~\ref{prob:3}
for $\HS=\RR^2$, $\upalpha=200$, and where $\mathsf{S}$ is
generated randomly. For every $\yS\in\mathsf{S}$, the set
$\menge{(\zS,\upxi)\in\HS\times\RR}
{\norm{\zS-\yS}_{\HS}\leq\upxi}$ corresponds to the translated
second-order cone and its projector operator is known explicitly
\cite[Proposition~3.3]{Fuku02}. We employ Theorem~\ref{t:2} with
$\mathsf{M}=1$ and $\lambda_{\nS}\equiv 1$.

\section*{Acknowledgement}

This work is part of the author's Ph.D. dissertation.
The author gratefully acknowledges the guidance of his Ph.D.
advisor P. L. Combettes throughout this work.


\begin{thebibliography}{99}
\setlength{\itemsep}{0pt}
\small

\bibitem{Alim22}
A. R. Alimov and I. G. Tsar'kov,
{\em Geometric Approximation Theory}.
Springer, Cham, 2022.

\bibitem{Nfao15}
A. Alotaibi, P. L. Combettes, and N. Shahzad,
Best approximation from the Kuhn--Tucker set of composite 
monotone inclusions,
{\em Numer. Funct. Anal. Optim.},
vol. 36, pp. 1513--1532, 2015.

\bibitem{Mor1}
H. H. Bauschke and P. L. Combettes, 
A weak-to-strong convergence principle for Fej\'er-monotone
methods in Hilbert spaces,
{\em Math. Oper. Res.},
vol. 26, pp. 248--264, 2001.

\bibitem{Livre1} 
H. H. Bauschke and P. L. Combettes, 
{\em Convex Analysis and Monotone Operator Theory in Hilbert 
Spaces}, 2nd ed. 
Springer, New York, 2017.

\bibitem{Sadd22}
M. N. B\`ui and P. L. Combettes,
Multivariate monotone inclusions in saddle form,
{\em Math. Oper. Res.},
vol. 47, pp. 1082--1109, 2022.

\bibitem{sicon1}
P. L. Combettes,
Strong convergence of block-iterative outer approximation methods
for convex optimization,
{\em SIAM J. Control Optim.},
vol. 38, pp. 538--565, 2000.

\bibitem{Else01} 
P. L. Combettes,
Quasi-Fej\'erian analysis of some optimization algorithms, in:
{\em Inherently Parallel Algorithms for Feasibility and 
Optimization}, 
(D. Butnariu, Y. Censor, and S. Reich, eds.),
pp. 115--152. Elsevier, New York, 2001.

\bibitem{Sign03} 
P. L. Combettes, 
A block-iterative surrogate constraint splitting method for 
quadratic signal recovery,
{\em IEEE Trans. Signal Process.}, 
vol. 51, pp. 1771--1782, 2003.

\bibitem{Acnu24} 
P. L. Combettes,
The geometry of monotone operator splitting methods,
{\em Acta Numer.},
vol. 33, pp. 487--632, 2024.

\bibitem{MaPr18}
P. L. Combettes and J. Eckstein,
Asynchronous block-iterative primal-dual decomposition methods
for monotone inclusions,
{\em Math. Program.},
vol. B168, pp. 645--672, 2018.

\bibitem{Siim25}
P. L. Combettes and J. I. Madariaga,
Almost-surely convergent randomly activated monotone operator
splitting methods,
{\em SIAM J. Imaging Sci.},
vol. 18, pp. 2177--2205, 2025.

\bibitem{Moco25}
P. L. Combettes and J. I. Madariaga,
A geometric framework for stochastic iterations,
arxiv, 2025. \\
\url{https://arxiv.org/pdf/2504.02761}

\bibitem{Siop15} 
P. L. Combettes and J.-C. Pesquet, 
Stochastic quasi-Fej\'er block-coordinate fixed point iterations
with random sweeping,
{\em SIAM J. Optim.}, 
vol. 25, pp. 1221--1248, 2015.

\bibitem{jat4}
P. L. Combettes and Z. C. Woodstock, 
Reconstruction of functions from prescribed proximal points,
{\em J. Approx. Theory},
vol. 268, art. 105606, 26 pp., 2021.

\bibitem{Dieu23}
A. Dieuleveut, G. Fort, E. Moulines, and H.-T. Wai, 
Stochastic approximation beyond gradient for signal processing 
and machine learning, 
{\em IEEE Trans. Signal Process.}, 
vol. 71, pp. 3117--3148, 2023.

\bibitem{Doob53}
J. L. Doob,
{\em Stochastic Processes}.
Wiley, New York, 1953.

\bibitem{Fuku02}
M. Fukushima, Z.-Q. Luo, and P. Tseng,
Smoothing functions for second-order-cone complementarity problems,
{\em SIAM J. Optim.},
vol. 12, pp. 436--460, 2002.

\bibitem{Haug68} 
Y. Haugazeau, 
{\em Sur les In\'equations Variationnelles et la Minimisation de 
Fonctionnelles Convexes}.
Th\`ese, Universit\'e de Paris, 1968. 

\bibitem{Herm19}
N. Hermer, D. R. Luke, and A. Sturm,
Random function iterations for consistent stochastic feasibility,
{\em Numer. Funct. Anal. Optim.},
vol. 40, pp. 386--420, 2019.

\bibitem{Hyto16}
T. Hyt\"onen, J. van Neerven, M. Veraar, and L. Weis,
{\em Analysis in Banach Spaces. Volume I: Martingales and 
Littlewood--Paley Theory}.
Springer, New York, 2016.

\bibitem{John24}
P. R. Johnstone, J. Eckstein, T. Flynn, and S. Yoo,
Stochastic projective splitting,
{\em Comput. Optim. Appl.},
vol. 87, pp. 397--437, 2024.

\bibitem{Kost23}
V. R. Kosti\'c and S. Salzo,
The method of randomized Bregman projections for stochastic
feasibility problems,
{\em Numer. Algorithms}, 
vol. 93, pp. 1269--1307, 2023.

\bibitem{Ledo91}
M. Ledoux and M. Talagrand,
{\em Probability in Banach Spaces: Isoperimetry and Processes}.
Springer, New York, 1991.

\bibitem{Luke26}
D. R. Luke, S. Schultze, and H. Grubm\"uller,
Stochastic algorithms for large-scale composite optimization: the
case of likelihood maximization for X-FEL imaging,
{\em Math. Program.},
to appear.

\bibitem{Marq25}
M. Marques Alves, J. E. Navarro Caballero, and R. T. Marcavillaca,
An inexact inertial projective splitting algorithm with strong
convergence,
{\em J. Optim. Theory Appl.},
vol. 207, art. 65, 2025.

\bibitem{Marq26}
M. Marques Alves, J. E. Navarro Caballero, M. Geremia, and 
R. T. Marcavillaca,
A strongly convergent inertial inexact proximal-point algorithm for
monotone inclusions with applications to variational inequalities,
{\em Optimization},
vol. 75, pp. 543--570, 2026.

\bibitem{Neco21}
I. Necoara and A. Nedi\'c, 
Minibatch stochastic subgradient-based projection algorithms for
feasibility problems with convex inequalities,
{\em Comput. Optim. Appl.},
vol. 80, pp. 121--152, 2021. 

\bibitem{Pier76}
G. Pierra,
\'Eclatement de contraintes en parall\`ele pour la minimisation
d'une forme quadratique,
{\em Lecture Notes in Comput. Sci.},
vol. 41. Springer, New York, 1976, pp. 200--218.

\bibitem{Solo00}
M. V. Solodov and B. F. Svaiter,
Forcing strong convergence of proximal point iterations in a
Hilbert space,
{\em Math. Program.},
vol. A87, pp. 189-202, 2000.

\end{thebibliography}
\end{document}